\documentclass[a4paper,
fontsize=11pt,%
oneside,%
numbers=enddot]{scrartcl}
\KOMAoptions{DIV=11}   % size of type area (see user guide)
\usepackage[utf8]{inputenc}
\usepackage[T1]{fontenc}
\usepackage{graphicx}
\usepackage{textcomp}
\usepackage[fleqn]{amsmath}
\usepackage{amsthm}
\usepackage{amssymb}
\usepackage{amsfonts}
\usepackage{pifont}        %for dingolist
\usepackage{libertine}
\usepackage[libertine,
slantedGreek,
nosymbolsc,
nonewtxmathopt,
subscriptcorrection]{newtxmath}
\usepackage[scaled=0.95,varqu,varl]{inconsolata}
\usepackage[scr=boondox]{mathalpha}   % for \mathscr
\usepackage{euscript}   % for \EuScript
\usepackage{leftindex}
\allowdisplaybreaks[1]  % allows align to split
\numberwithin{equation}{section}    % *must* be *before* cleveref!
\usepackage[svgnames,hyperref]{xcolor}
\definecolor{orng}{HTML}{F35400}
\definecolor{bleu}{HTML}{BCE6F2}
\definecolor{dblue}{HTML}{0455BF}
\definecolor{dgreen}{HTML}{02724A}
\definecolor{dgreen2}{HTML}{025951}
\definecolor{dred}{HTML}{D90404}
\definecolor{dviolet}{HTML}{42208C}
\definecolor{labelkey}{HTML}{025951}
\definecolor{refkey}{HTML}{025951}
\definecolor{refkey}{rgb}{0,0.6,0.0}
\definecolor{Brown}{rgb}{0.45,0.0,0.05}
\definecolor{dgreen}{rgb}{0.00,0.49,0.00}
\definecolor{dblue}{rgb}{0,0.18,0.75}
\definecolor{lblue}{rgb}{0,0.7,0.75}
\definecolor{dviolet}{HTML}{9400D3}
\definecolor{pblue}{rgb}{0.1176,0.5647,1}
\definecolor{nblue}{rgb}{0.2,0.3,1}
\definecolor{pgreen}{rgb}{0.1961,0.8039,0.1961}
\definecolor{ngreen}{rgb}{0.0,0.6,0.3}
\definecolor{pred}{rgb}{1.0,0.2706,0.0}
\definecolor{magenta}{HTML}{ff00ff}
\usepackage{tikz,tkz-euclide,pgfplots}
\usetikzlibrary{arrows,decorations}
\addtokomafont{section}{\centering}

\usepackage{enumitem}
\setlist{itemsep=-2.0pt}

\usepackage{upref}  % upright number when using ref in theorems
\usepackage{hyperref}
\hypersetup{colorlinks=true,
linktocpage=true,
linkcolor=dblue,
citecolor=dgreen,
urlcolor=dred,
pdfencoding=auto,
hypertexnames=false}
\makeatletter
\g@addto@macro\th@plain{
\thm@headfont{\bfseries\sffamily}
\thm@notefont{}}
\g@addto@macro\th@definition{
\thm@headfont{\bfseries\sffamily}
\thm@notefont{}}
\g@addto@macro\th@remark{
\thm@headfont{\bfseries\sffamily}
\thm@notefont{}}
\makeatother
\theoremstyle{plain}
\newtheorem{theorem}{Theorem}[section]
\newtheorem{proposition}[theorem]{Proposition}
\newtheorem{corollary}[theorem]{Corollary}
\newtheorem{lemma}[theorem]{Lemma}

\theoremstyle{definition}

\newtheorem{example}[theorem]{Example}
\newtheorem{problem}[theorem]{Problem}

\theoremstyle{remark}
\newtheorem{remark}[theorem]{Remark}

\usepackage[extdef=true]{delimset}
\DeclareMathDelimiterSet{\scal}[2]{
\selectdelim[l]<{#1}
\mathpunct{}\selectdelim[p]|
{#2}\selectdelim[r]>}
\DeclareMathDelimiterSet{\EC}[2]{
\mathsf{E}\selectdelim[l]({#1}
\mathpunct{}\selectdelim[p]|
{#2}\selectdelim[r])}

\newcommand{\menge}[2]{\bigl\{{#1}\mid{#2}\bigr\}} 
\DeclareMathDelimiterSet{\Menge}[2]{\selectdelim[l]\{
{#1}\selectdelim[m]|{#2}\selectdelim[r]\}}
\makeatletter
\def\upintkern@{\mkern-7mu\mathchoice{\mkern-3.5mu}{}{}{}}
\def\upintdots@{\mathchoice{\mkern-4mu\@cdots\mkern-4mu}%
{{\cdotp}\mkern1.5mu{\cdotp}\mkern1.5mu{\cdotp}}%
{{\cdotp}\mkern1mu{\cdotp}\mkern1mu{\cdotp}}%
{{\cdotp}\mkern1mu{\cdotp}\mkern1mu{\cdotp}}}
\makeatother
\DeclareFontFamily{OMX}{mdbch}{}
\DeclareFontShape{OMX}{mdbch}{m}{n}{ <->s * [0.8]  mdbchr7v }{}
\DeclareFontShape{OMX}{mdbch}{b}{n}{ <->s * [0.8]  mdbchb7v }{}
\DeclareFontShape{OMX}{mdbch}{bx}{n}{<->ssub * mdbch/b/n}{}
\DeclareSymbolFont{uplargesymbols}{OMX}{mdbch}{m}{n}
\SetSymbolFont{uplargesymbols}{bold}{OMX}{mdbch}{b}{n}
\DeclareMathSymbol{\upintop}{\mathop}{uplargesymbols}{82}
\DeclareMathSymbol{\upointop}{\mathop}{uplargesymbols}{"48}
\makeatletter
\renewcommand{\int}{\DOTSI\upintop\ilimits@}
\renewcommand{\oint}{\DOTSI\upointop\ilimits@}
\makeatother
\newcommand{\RR}{\mathbb{R}}

\newcommand{\KKS}{\ensuremath{\boldsymbol{\mathsf K}}}

\newcommand{\XXS}{\ensuremath{\boldsymbol{\mathsf X}}}
\newcommand{\YYS}{\ensuremath{\boldsymbol{\mathsf{Y}}}}
\newcommand{\GGS}{\ensuremath{{\boldsymbol{\mathsf G}}}}
\newcommand{\XS}{\mathsf{X}}
\newcommand{\YS}{\mathsf{Y}}
\newcommand{\HS}{\mathsf{H}}
\newcommand{\GS}{\mathsf{G}}
\newcommand{\ZS}{\mathsf{Z}}
\newcommand{\nS}{{\mathsf{n}}}
\newcommand{\nnn}{\mathsf{n}\in\mathbb{N}}
\newcommand{\iS}{\mathsf{i}}

\newcommand{\jS}{\mathsf{j}}
\newcommand{\kS}{\mathsf{k}}
\newcommand{\lS}{\mathsf{l}}
\newcommand{\LS}{\mathsf{L}}
\newcommand{\VS}{\mathsf{S}}
\newcommand{\LLS}{\boldsymbol{\mathsf{L}}}
\newcommand{\BE}{\EuScript{B}}
\newcommand{\FE}{\EuScript{F}}

\newcommand{\pinf}{{+}\infty}
\newcommand{\minf}{{-}\infty}
\newcommand{\zeroun}{\intv[o]{0}{1}}
\newcommand{\rzeroun}{\intv[l]{0}{1}}

\newcommand{\RX}{\intv[l]0{\minf}{\pinf}}

\newcommand{\RPP}{\intv[o]0{0}{\pinf}}

\newcommand{\emp}{\varnothing}

\newcommand{\Sum}{\displaystyle\sum}

\newcommand{\Frac}[2]{\displaystyle{\dfrac{#1}{#2}}} 
\newcommand{\minimize}[2]{\underset{\substack{{#1}}}
{\operatorname{minimize}}\;\;#2}

\newcommand{\pushfwd}%
{\ensuremath{\mbox{\Large$\,\triangleright\,$}}}

\DeclareMathOperator{\Argmin}{Argmin}

\newcommand{\Id}{\mathsf{Id}}
\newcommand{\ID}{\boldsymbol{\mathsf{Id}}}

\DeclareMathOperator{\card}{card}

\DeclareMathOperator{\dom}{dom}

\DeclareMathOperator{\gra}{gra}
\DeclareMathOperator{\zer}{zer}

\DeclareMathOperator{\reli}{ri}

\DeclareMathOperator{\prox}{prox}
\DeclareMathOperator{\proj}{proj}

\newcommand{\PP}{\mathsf{P}}

\DeclareMathOperator{\dft}{DFT}
\DeclareMathOperator{\idft}{IDFT}
\renewcommand{\leq}{\leqslant}
\renewcommand{\geq}{\geqslant}
\newcommand{\exi}{\exists\,}
\newcommand{\weakly}{\rightharpoonup}

\newcommand{\Pas}{\text{\normalfont$\PP$-a.s.}}
\renewenvironment{abstract}{%
\vspace*{-0.50cm}
\small
\quotation%
\noindent%
{\normalfont\bfseries\sffamily
\nobreak\abstractname\ }%
}{%
\endquotation%
\medskip
}
\renewcommand{\abstractname}{Abstract.}
\newcommand\keywordsname{Keywords.}
\newenvironment{keywords}
{\renewcommand\abstractname{\keywordsname}\begin{abstract}}
{\end{abstract}}
\usepackage[auth-sc]{authblk}
\newcommand{\email}[1]{\href{mailto:#1}{\nolinkurl{#1}}}
\renewcommand*\Affilfont{\normalfont\normalsize}
\newcommand\affilcr{\protect\\ \protect\Affilfont}
\makeatletter
\renewcommand\AB@affilsepx{\protect\\[0.5em]}
\makeatother
\author[1]{Patrick L. Combettes}
\affil[1]{North Carolina State University
\affilcr
Department of Mathematics
\affilcr
Raleigh, NC 27695, USA
\affilcr
\email{plc@math.ncsu.edu}
}
\author[2]{Javier I. Madariaga}
\affil[2]{North Carolina State University
\affilcr
Department of Mathematics
\affilcr
Raleigh, NC 27695, USA
\affilcr
\email{jimadari@ncsu.edu}
}

\begin{document}

\title{Almost-Surely Convergent Randomly Activated Monotone 
Operator Splitting Methods\thanks{Contact author: P. L.
Combettes. Email: \email{plc@math.ncsu.edu}. 
This work was supported by the National
Science Foundation under grant CCF-2211123.
}}

\date{~}

\maketitle

\begin{abstract}
We propose stochastic splitting algorithms for solving large-scale
composite inclusion problems involving monotone and linear
operators. They activate at each iteration blocks of randomly
selected resolvents of monotone operators and, unlike existing
methods, achieve almost sure convergence of the iterates to a
solution without any regularity assumptions or knowledge of the
norms of the linear operators. Applications to image recovery and
machine learning are provided.
\end{abstract}


\newpage

\section{Introduction}
\label{sec:1}

The problem of extracting information from data is at the core of
many tasks in signal processing, inverse problems, and machine
learning. A prevalent methodology to seek meaningful solutions is
to build a mathematical model that incorporates the prior knowledge
about the object of interest $\overline{\mathsf{x}}$ and the data, 
which consist of observations mathematically or physically related 
to $\overline{\mathsf{x}}$ (see Figure~\ref{fig:1}). 	
\begin{figure}[b!]
\begin{center}
\scalebox{0.8} % Change this value to rescale the drawing.
{
\begin{tikzpicture}
\draw[line width=0.05cm, fill=bleu] (5,5.1) ellipse (1.75
and 0.5);
\draw[line width=0.05cm, fill=bleu] (13,5.1) ellipse (1.75
and 0.5);
\draw[line width=0.05cm, fill=bleu] (9,3.4) ellipse (2.05
and 0.6);
\draw[line width=0.05cm, fill=bleu] (9,1.6) ellipse (1.75
and 0.5);
\draw[line width=0.05cm, fill=bleu] (9,-0.2) ellipse (1.75
and 0.5);
\node at (5,5.1) {Prior knowledge};
\node at (13,5.1) {Observations};
\node at (9,3.4) {Mathematical model};
\node at (9,1.6) {Algorithm};
\node at (9,-0.2) {Solution};
\draw[line width=0.05cm,->] (6.0,4.6) -- (7.2,3.75);
\draw[line width=0.05cm,->] (12.0,4.6) -- (10.8,3.75);
\draw[line width=0.05cm,->] (9.0,2.7) -- (9.0,2.2);
\draw[line width=0.05cm,->] (9.0,1.0) -- (9.0,0.4);
\end{tikzpicture}
}
\end{center}
\caption{Data processing flowchart.}
\label{fig:1}
\end{figure}
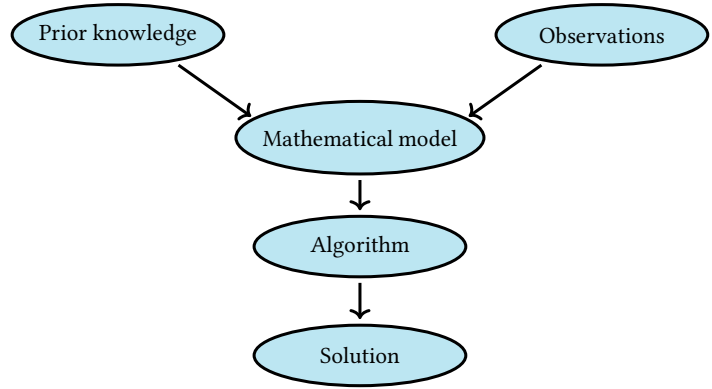
Since the first mathematical formalizations of Euler \cite{Eule49}
and Mayer \cite{Maye48} in the late 1740s, which contained the
embryo of least-squares data fitting techniques, convex
minimization formulations have been a tool of choice. The following
problem encapsulates a broad range of minimization models found in
data analysis problems
\cite{Andr77,Bach12,Bane20,Benn18,Cham16,Banf11,Smms05,%
Glow16,Judi20,Theo20} (see Section~\ref{sec:21} for notation).

\begin{problem}
\label{prob:10}
$\HS$ is a separable real Hilbert space and 
$\mathsf{f}\in\upGamma_0(\HS)$. For every
$\kS\in\{1,\ldots,\mathsf{p}\}$, $\GS_{\kS}$ is a separable real
Hilbert space, $\mathsf{g}_{\kS}\in\upGamma_0(\GS_{\kS})$, and
$0\neq\LS_{\kS}\colon\HS\to\GS_{\kS}$ is linear and bounded. It is
assumed that $\zer(\partial\mathsf{f}+\sum_{\kS=1}^\mathsf{p}
\LS^*_{\kS}\circ(\partial\mathsf{g}_{\kS})\circ\LS_{\kS})\neq\emp$. 
The task is to
\begin{equation}
\label{e:p10}
\minimize{\mathsf{x}\in\HS}{\mathsf{f}(\mathsf{x})+
\sum_{\kS=1}^\mathsf{p}\mathsf{g}_{\kS}
(\LS_{\kS}\mathsf{x})}.
\end{equation}
\end{problem}

In recent years, an increasing number of problem formulations have
emerged, which cannot be naturally reduced to tractable
minimization problems and which are best captured by more general
notions of equilibria provided by inclusion problems
\cite{Svva20,Sign21,Joat21,Siim22,Flor23,Judi19,Pesq21,Wins20,%
Yich20}. 
A formulation covering such models, as well as
Problem~\ref{prob:10}, is the following composite monotone
inclusion formulation.

\begin{problem}
\label{prob:1}
$\HS$ is a separable real Hilbert space and 
$\mathsf{A}\colon\HS\to 2^{\HS}$ is maximally monotone.
For every $\kS\in\{1,\ldots,\mathsf{p}\}$, $\GS_{\kS}$ is a 
separable real Hilbert space, 
$\mathsf{B}_{\kS}\colon\GS_{\kS}\to 2^{\GS_{\kS}}$ is maximally
monotone, and $0\neq\LS_{\kS}\colon\HS\to\GS_{\kS}$ is linear and
bounded. It is assumed that $\mathsf{Z}=\zer(\mathsf{A}+
\sum_{\kS=1}^\mathsf{p}\LS_{\kS}^*\circ\mathsf{B}_{\kS}\circ
\LS_{\kS})\neq\emp$. The task is to
\begin{equation}
\label{e:p1}
\text{find}\;\;\mathsf{x}\in\HS\;\;\text{such that}\;\;
\mathsf{0}\in\mathsf{A}\mathsf{x}+\sum_{\kS=1}^\mathsf{p}
\LS_{\kS}^*\brk1{\mathsf{B}_{\kS}(\LS_{\kS}\mathsf{x})}.
\end{equation}
\end{problem}

Splitting algorithms for solving Problem~\ref{prob:1} operate on
the principle that each nonlinear and linear operator is used
separately over the course of the iterations. Since the nonlinear
operators are general set-valued monotone operators, they must be
activated through their resolvent. Various deterministic operator
splitting methods are available to solve Problem~\ref{prob:1}, most
of which require the activation of the resolvents of the
$\mathsf{p}+1$ operators $\mathsf{A}$ and
$(\mathsf{B}_{\kS})_{1\leq\kS\leq\mathsf{p}}$ at each iteration
\cite{Acnu24}. Our specific focus is on solving
Problem~\ref{prob:1} in instances when $\mathsf{p}$ is large, as is
often the case in data analysis problems. In
such scenarios, memory and computing power limitations make the
execution of standard monotone operator splitting algorithms
inefficient if not simply impossible. We aim at designing monotone
splitting algorithms which are stochastic in the sense that they
activate a randomly selected block of operators at each iteration
and, in addition, allow for random errors in the implementation of
these resolvent steps. Furthermore, the proposed algorithms satisfy
the following requirements: 
\begin{itemize} 
\item[\bfseries R1:]
They guarantee the almost sure convergence of the sequence of 
iterates to a solution to Problem~\ref{prob:1} (respectively
Problem~\ref{prob:10}) without any additional assumptions
on the nonlinear operators (respectively the functions), the linear
operators, or the underlying Hilbert spaces. 
\item[\bfseries R2:] 
At each iteration, more than one
randomly selected resolvent of the operators
$(\mathsf{A},\mathsf{B}_1,\ldots,\mathsf{B}_{\mathsf{p}})$ can be
activated. 
\item[\bfseries R3:] 
Knowledge of bounds on the norms of the linear operators is not
required. 
\item[\bfseries R4:]
The operators are available only through a stochastic 
approximation. 
\end{itemize} 
Requirement {\bfseries R1} imposes actual iterate convergence to a
solution and not a weaker form of convergence such as ergodic
convergence, vanishing stepsizes, or, in the context of
Problem~\ref{prob:10}, convergence of the values of the objective
function. It also asks that Problems~\ref{prob:1} and
\ref{prob:10} be addressed in their generality, without
restricting their scope by introducing additional assumptions.
Requirement {\bfseries R2} makes it possible to activate more than
one operator, hence opening the way to matching efficiently the
computational load of an iteration to the possibly parallel
architecture at hand. Requirement {\bfseries R3} broadens the 
scope of the methods by not assuming any knowledge of the norms of 
the linear operators present in the model. For instance, in domain
decomposition methods, it is quite difficult to obtain tight
upper bounds on the norms of the trace operators \cite{Nume16}. 
Finally, in the spirit of the classical stochastic iteration models
of \cite{Blum54,Dufl97,Robi51}, {\bfseries R4} addresses the
robustness of the algorithm to stochastic errors affecting the
implementation of the operators.

As will be seen in the literature review of Section~\ref{sec:22},
there does not seem to exist methods that satisfy simultaneously
{\bfseries R1--R4}. Our main contribution is presented in
Section~\ref{sec:3}, where we propose three algorithmic frameworks
that comply with {\bfseries R1--R4}. Section~\ref{sec:4} is 
devoted to the minimization setting of Problem~\ref{prob:10}. The 
last section of the paper is Section~\ref{sec:5}, where the 
proposed algorithms are applied to signal restoration, support 
vector machine, classification, and image reconstruction problems.

\section{Notation and existing algorithms}
\label{sec:2}

\subsection{Notation}
\label{sec:21}

Throughout, $\HS$ is a separable real Hilbert space with power set
$2^{\HS}$, identity operator $\Id$, scalar product 
$\scal{\cdot}{\cdot}$, and associated norm $\|\cdot\|$. 

Let $\mathsf{A}\colon\HS\to 2^{\HS}$. The graph of $\mathsf{A}$ is 
$\gra \mathsf{A}=\menge{(\mathsf{x},\mathsf{x}^*)\in\HS\times\HS}
{\mathsf{x}^*\in\mathsf{Ax}}$ and the set of zeros of $\mathsf{A}$
is $\zer\mathsf{A}=\menge{\mathsf{x}\in\HS}
{0\in\mathsf{A}\mathsf{x}}$. The inverse of
$\mathsf{A}$ is the operator 
$\mathsf{A}^{-1}\colon\HS\to 2^{\HS}$ with graph
$\gra\mathsf{A}^{-1}=\menge{(\mathsf{x}^*,\mathsf{x})
\in\HS\times\HS}{\mathsf{x}^*\in\mathsf{Ax}}$ and
the resolvent of $\mathsf{A}$ is
$\mathsf{J}_{\mathsf{A}}=(\Id+\mathsf{A})^{-1}$.
Further, $\mathsf{A}$ is maximally monotone if
\begin{equation}
\label{e:m2}
\big(\forall(\mathsf{x},\mathsf{x}^*)
\in\HS\times\HS\big)\quad
\bigl[\:(\mathsf{x},\mathsf{x}^*)\in\gra\mathsf{A}\;
\Leftrightarrow\;\bigl(\forall(\mathsf{y},\mathsf{y}^*)\in
\gra\mathsf{A}\bigr)\;\;\scal{\mathsf{x}-\mathsf{y}}
{\mathsf{x}^*-\mathsf{y}^*}\geq 0\:\bigr].
\end{equation}
An operator $\mathsf{F}\colon\HS\to\HS$ is firmly nonexpansive if
\begin{equation}
(\forall\mathsf{x}\in\HS)(\forall\mathsf{y}\in\HS)
\quad\scal{\mathsf{x}-\mathsf{y}}
{\mathsf{F}\mathsf{x}-\mathsf{F}\mathsf{y}}\geq
\|{\mathsf{F}\mathsf{x}-\mathsf{F}\mathsf{y}}\|^2.
\end{equation}

\begin{lemma}
\label{l:600}
Let $\mathsf{F}\colon\HS\to\HS$ be firmly nonexpansive and let
$\upgamma\in\RPP$. Then there exists a maximally monotone operator 
$\mathsf{A}\colon\HS\to 2^{\HS}$ such that the following hold:
\begin{enumerate}
\item
\label{l:600i}
$\mathsf{F}=\mathsf{J}_{\mathsf{A}}$.
\item
\label{l:600ii}
$\mathsf{J}_{\upgamma\mathsf{F}}=\Id-\upgamma
\mathsf{J}_{(1+\upgamma)^{-1}\mathsf{A}}\circ
(1+\upgamma)^{-1}\Id$.
\end{enumerate}
\end{lemma}
\begin{proof}
\ref{l:600i}: See \cite[Corollary~23.9]{Livre1}.

\ref{l:600ii}: This follows from \ref{l:600i} and
\cite[Proposition~23.29]{Livre1}.
\end{proof}

$\upGamma_0(\HS)$ denotes the class of lower semicontinuous convex
functions $\mathsf{f}\colon\HS\to\RX$ such that 
$\dom\mathsf{f}=\menge{\mathsf{x}\in\HS}
{\mathsf{f}(\mathsf{x})<\pinf}\neq\emp$. Let
$\mathsf{f}\in\upGamma_0(\HS)$. The subdifferential of $\mathsf{f}$
is the maximally monotone operator
\begin{equation}
\label{e:subdiff}
\partial\mathsf{f}\colon\HS\to 2^{\HS}\colon\mathsf{x}\mapsto
\menge{\mathsf{x}^*\in\HS}{(\forall\mathsf{z}\in\HS)\;
\scal{\mathsf{z}-\mathsf{x}}{\mathsf{x}^*}+\mathsf{f}(\mathsf{x})
\leq\mathsf{f}(\mathsf{z})} 
\end{equation}
and the proximity operator of $\mathsf{f}$ is 
\begin{equation}
\label{e:dprox}
\prox_{\mathsf{f}}=\mathsf{J}_{\partial\mathsf{f}}\colon\HS\to\HS
\colon\mathsf{x}\mapsto\underset{\mathsf{z}\in\HS}
{\text{argmin}}\;\brk2{\mathsf{f}(\mathsf{z})+
\dfrac{1}{2}\|\mathsf{x}-\mathsf{z}\|^2}.
\end{equation}
Let $\mathsf{C}$ be a nonempty closed convex subset of $\HS$. Then 
$\iota_{\mathsf{C}}$ denotes the indicator function of 
$\mathsf{C}$, $\mathsf{d}_{\mathsf{C}}$ the distance function to
$\mathsf{C}$, 
\begin{equation}
\label{e:nc}
\mathsf{N}_{\mathsf{C}}=\partial\iota_{\mathsf{C}}\colon
\mathsf{x}\mapsto
\begin{cases}
\menge{\mathsf{x}^*\in\HS}{(\forall\mathsf{y}\in\mathsf{C})\;\;
\scal{\mathsf{y}-\mathsf{x}}{\mathsf{x}^*}\leq 0},
&\text{if}\;\;\mathsf{x}\in\mathsf{C};\\
\emp,&\text{otherwise}
\end{cases}
\end{equation}
the normal cone operator of $\mathsf{C}$, and
$\proj_{\mathsf{C}}=\prox_{\iota_{\mathsf{C}}}
=\mathsf{J}_{\mathsf{N}_{\mathsf{C}}}$
the projection operator onto $\mathsf{C}$. In particular, if
$\mathsf{V}$ is a closed vector subspace of $\HS$, 
\begin{equation}
\label{e:nv}
\mathsf{N}_{\mathsf{V}}\colon\HS\to2^{\HS}\colon\mathsf{x}\mapsto
\begin{cases}
\mathsf{V}^\bot,&\text{if}\;\;\mathsf{x}\in\mathsf{V};\\
\emp,&\text{otherwise.}
\end{cases}
\end{equation}

The underlying probability space $(\upOmega,\FE,\PP)$ is assumed to
be complete and $\BE_{\HS}$ denotes the Borel $\upsigma$-algebra of
$\HS$. An $\HS$-valued random variable is a measurable 
mapping $x\colon(\upOmega,\FE)\to(\HS,\BE_{\HS})$. 
The $\upsigma$-algebra generated by a family $\upPhi$ of 
random variables is denoted by $\upsigma(\upPhi)$.
Given $x\colon\upOmega\to\HS$ and $\mathsf{S}\subset\HS$, we set 
$[x\in\mathsf{S}]=\menge{\upomega\in\upOmega}{x(\upomega)\in
\mathsf{S}}$.
The reader is referred to \cite{Livre1} for background on monotone
operators and convex analysis, and to \cite{Ledo91} for background
on probability in Hilbert spaces. 

We use sans-serif letters to denote deterministic variables and
italicized serif letters to denote random variables. Finally, in
connection with Problem~\ref{prob:1}, we define the Hilbert direct
sum
\begin{equation}
\label{e:G}
\GGS=\GS_1\oplus\cdots\oplus\GS_{\mathsf{p}},
\end{equation}
as well as the subspace
\begin{equation}
\label{e:2024}
\boldsymbol{\mathsf{W}}=
\Menge3{\boldsymbol{\mathsf{x}}\in\HS\oplus\GGS}
{(\forall\kS\in\{1,\ldots,\mathsf{p}\})\;\mathsf{x}_{\kS+1}=
\LS_{\kS}\mathsf{x}_{1}},
\end{equation}
and note that 
\begin{equation}
\label{e:2024d}
\boldsymbol{\mathsf{W}}^\bot=
\Menge3{\boldsymbol{\mathsf{x}}^*\in\HS\oplus\GGS}
{\mathsf{x}_1^*=-\sum_{\kS=1}^{\mathsf{p}}
\LS_{\kS}^*\mathsf{x}_{\kS+1}^*}.
\end{equation}

\subsection{Existing algorithms}
\label{sec:22}

It seems that no algorithm satisfying requirements 
{\bfseries R1--R4} has been explicitly proposed to solve
Problem~\ref{prob:1} --- or even Problem~\ref{prob:10} --- in the
literature. There is a vast body of papers on random activation
algorithms in the special case of Problem~\ref{prob:10} that
consists in minimizing a sum of smooth functions
$\sum_{\kS=1}^\mathsf{p}\mathsf{g}_{\kS}$ in $\HS=\RR^\mathsf{N}$
via so-called stochastic gradient descent methods. Their principle
is to activate a randomly selected gradient in 
$(\nabla\mathsf{g}_{\kS})_{1\leq\kS\leq\mathsf{p}}$ at each
iteration; see \cite{Dieu23} and its bibliography and 
\cite{Con23a,Trao24}
for related work with random proximal activations for this type
problem. These methods focus on a very specific instance of
Problem~\ref{prob:10} and they do not satisfy 
{\bfseries R1}--{\bfseries R2}. The only random activation 
algorithm tailored to Problem~\ref{prob:10} which 
guarantees almost sure convergence of the iterates without
additional assumptions such as strong convexity is the following
(see also \cite{Alac22} for a non-adaptive version).

\begin{proposition}[{\cite[Theorem~2.1 and Algorithm~3.1]{Cham24}}]
\label{p:72}
Consider the setting of Problem~\ref{prob:10} and suppose that
$\HS=\RR^{\mathsf{N}}$ and, for every 
$\kS\in\{1,\dots,\mathsf{p}\}$,
$\GS_{\kS}=\RR^{\mathsf{M}_{\kS}}$, all considered as standard
Euclidean spaces. Let 
$(\uppi_{\kS})_{1\leq\kS\leq\mathsf{p}}$ be real numbers in 
$\rzeroun$ such 
that $\sum_{\kS=1}^{\mathsf{p}}\uppi_{\kS}=1$, and let 
$(k_{\nS})_{\nnn}$ be identically distributed 
$\{1,\dots,\mathsf{p}\}$-valued random variables such that, for 
every $\kS\in\{1,\dots,\mathsf{p}\}$, $\PP[k_0=\kS]=\uppi_{\kS}$.
Set, for every $\kS\in\{1,\dots,\mathsf{p}\}$ and every $\nnn$,
$\varepsilon_{\kS,\nS}=\mathsf{1}_{[k_{\nS}=\kS]}$.
Let $\tau_0\in\RPP$ and $\sigma_0\in\RPP$ be such that
\begin{equation}
\label{e:CHcond}
\tau_0\sigma_0\max_{1\leq\kS\leq\mathsf{p}}
\frac{\|\LS_{\kS}\|^2}{\uppi_{\kS}}<1.
\end{equation}
Further, let $\chi_0\in\left[0,1\right[$, $\upeta\in\zeroun$, and
$\updelta\in\left]1,\pinf\right[$, set $\rho_0=0$ and $\nu_0=0$,
let $x_{1,0}$ be a $\HS$-valued random variable, and let
$\boldsymbol{y}_0$ be a $\GGS$-valued random variable. 
Set ${\boldsymbol{z}}_0=\boldsymbol{y}_0$ and
$\boldsymbol{\LS}\colon\HS\to\GGS\colon\mathsf{x}
\mapsto(\LS_{\kS}\mathsf{x})_{1\leq\kS\leq\mathsf{p}}$, and 
iterate 
\begin{equation}
\label{e:72}
\begin{array}{l}
\text{for}\;\nS=0,1,\ldots\\
\left\lfloor
\begin{array}{l}
\brk1{\tau_{\nS+1},\sigma_{\nS+1},\chi_{\nS+1}}=
\begin{cases}
\brk2{\dfrac{\tau_{\nS}}{1-\chi_{\nS}},\sigma_{\nS}(1-\chi_{\nS}),
\chi_{\nS}\upeta},
&\text{if}\;\;\rho_{\nS}>\|\boldsymbol{\LS}\|\nu_{\nS}\updelta;\\
\brk2{\tau_{\nS}(1-\chi_{\nS}),\dfrac{\sigma_{\nS}}{1-\chi_{\nS}},
\chi_{\nS}\upeta},
&\text{if}\;\;\rho_{\nS}<\|\boldsymbol{\LS}\|\nu_{\nS}\updelta;\\
\brk1{\tau_{\nS},\sigma_{\nS},\chi_{\nS}},
&\text{if}\;\;\dfrac{\|\boldsymbol{\LS}\|\nu_{\nS}}{\updelta}\leq
\rho_{\nS}\leq\|\boldsymbol{\LS}\|\nu_{\nS}\updelta
\end{cases}\\[3mm]
{x}_{1,\nS+1}
=x_{1,\nS}+\prox_{\tau_{\nS+1}\mathsf{f}}\brk1{x_{1,\nS}-
\tau_{\nS+1}\sum_{\kS=1}^{\mathsf{p}}\LS_{\kS}^*{z}_{\kS,\nS}}\\
\text{for}\;\kS=1,\dots,\mathsf{p}\\
\left\lfloor
\begin{array}{l}
{y}_{\kS,\nS+1}
=y_{\kS,\nS+1}+\varepsilon_{\kS,\nS}\brk2{
\prox_{\sigma_{\nS+1}\mathsf{g}_{\kS}^*}\brk1{y_{\kS,\nS}+
\sigma_{\nS+1}\LS_{\kS}x_{1,\nS+1}}-y_{\kS,\nS}}\\[5mm]
{z}_{\kS,\nS+1}=
y_{\kS,\nS}+\varepsilon_{\kS,\nS}
\Bigl(y_{\kS,\nS+1}+\dfrac{1}{\uppi_{\kS}}\bigl(y_{\kS,\nS+1}
-y_{\kS,\nS}\bigr)-y_{\kS,\nS}\Bigr)\\[5mm]
\rho_{\nS+1}=\Bigl\|\dfrac{1}{\tau_{\nS+1}}
\bigl(x_{\nS}-x_{\nS+1}\bigr)-
\dfrac{1}{\uppi_{k_{\nS}}}\LS^*_{k_{\nS}}\bigl(y_{k_{\nS},\nS}-
y_{k_{\nS},\nS+1}\bigr)\Bigr\|_{1}\\[5mm]
\nu_{\nS+1}=\dfrac{1}{\uppi_{k_{\nS}}}\Bigl\|\LS_{k_{\nS}}
\bigl(x_{\nS}-x_{\nS+1}\bigr)-\dfrac{1}{\sigma_{\nS+1}}
\bigl(y_{k_{\nS},\nS}-y_{k_{\nS},\nS+1}\bigr)\Bigr\|_{1},\\
\end{array}
\right.\\
\end{array}
\right.\\
\end{array}
\end{equation}
where $\|\cdot\|_1$ denotes the $\ell^1$-norm.
Then $(x_{1,\nS})_{\nnn}$ converges $\Pas$ to an 
$\Argmin(\mathsf{f}+\sum_{\kS=1}^{\mathsf{p}}\mathsf{g}_{\kS}
\circ\LS_{\kS})$-valued random variable.
\end{proposition}

Algorithm \eqref{e:72} is of interest because it guarantees
{\bfseries R1} in a finite-dimensional setting. However, it does
not satisfy {\bfseries R2} since, at each
iteration, $\mathsf{f}$ must be activated together with one of the
functions $(\mathsf{g}_{\kS})_{1\leq\kS\leq\mathsf{p}}$. It does
not satisfy {\bfseries R3} either since it requires the knowledge
of the norms of linear operators in \eqref{e:CHcond}. 
We also note that it does not tolerate errors in the evaluation of
the proximity operators, which means that {\bfseries R4} is not
satisfied.

Let us now turn to the general Problem~\ref{prob:1}. The only
algorithm that satisfies {\bfseries R1} is that of \cite{Pesq15},
which corresponds to an implementation of the random
block-coordinate forward-backward algorithm of
\cite[Section~5.2]{Siop15} suggested in
\cite[Remark~5.10(iv)]{Siop15}. 

\begin{proposition}[{\cite[Proposition~4.6]{Pesq15}}]
\label{p:71}
Consider the setting of Problem~\ref{prob:1}. Let
$\mathsf{W}\colon\HS\to\HS$ and, for every
$\kS\in\{1,\dots,\mathsf{p}\}$,
$\mathsf{U}_{\kS}\colon\GS_{\kS}\to\GS_{\kS}$ be 
bounded linear strongly positive self-adjoint operators such that
\begin{equation}
\label{e:PRcond}
\Sum_{\kS=1}^{\mathsf{p}}
\bigl\|\mathsf{U}_{\kS}^{1/2}\LS_{\kS}\mathsf{W}^{1/2}\bigr\|^2 
<\dfrac{1}{2}. 
\end{equation}
Let $(\uplambda_{\nS})_{\nnn}$ be a sequence in $\left]0,1\right]$ 
such that $\inf_{\nnn}\uplambda_{\nS}>0$, let $x_{1,0}$ and 
$(a_{1,\nS})_{\nnn}$ be $\HS$-valued 
random variables, let $\boldsymbol{v}_0$ and 
$(\boldsymbol{b}_{\nS})_{\nnn}$ be $\GGS$-valued
random variables, and let
$(\boldsymbol{\varepsilon}_{\nS})_{\nnn}$ be identically
distributed $\{0,1\}^{\mathsf{p}}\smallsetminus
\{\boldsymbol{\mathsf{0}}\}$-valued random variables.
Iterate 
\begin{equation}
\label{e:71}
\begin{array}{l}
\text{for}\;\nS=0,1,\ldots\\
\left\lfloor
\begin{array}{l}
y_{1,\nS}=\mathsf{J}_{\mathsf{W}\mathsf{A}}
\Bigl(x_{1,\nS}-\mathsf{W}\bigl(\sum_{\kS=1}^{\mathsf{p}}
\LS^*_{\kS}v_{\kS,\nS}\bigr)\Bigr)+a_{1,\nS}\\
x_{1,\nS+1}=x_{1,\nS}+\uplambda_{\nS}(y_{1,\nS}-x_{1,\nS})\\
\text{for}\;\kS=1,\dots,\mathsf{p}\\
\left\lfloor
\begin{array}{l}
u_{\kS,\nS}=\varepsilon_{\kS,\nS}
\Bigl(\mathsf{J}_{\mathsf{U}_{\kS}\mathsf{B}_{\kS}^{-1}}
\Bigl(v_{\kS,\nS}+\mathsf{U}_{\kS}\bigl(
\LS_{\kS}(2y_{1,\nS}-x_{1,\nS})\bigr)\Bigr)+b_{\kS,\nS}\Bigr)\\
v_{\kS,\nS+1}={v}_{\kS,\nS}+\varepsilon_{\kS,\nS}\uplambda_{\nS}
\bigl(u_{\kS,\nS}-v_{\kS,\nS}\bigr),
\end{array}
\right.\\
\end{array}
\right.\\
\end{array}
\end{equation}
and set $(\forall\nnn)$
$\boldsymbol{\EuScript{E}}_{\nS}=
\upsigma(\boldsymbol{\varepsilon}_{\nS})$ and 
$\boldsymbol{\EuScript{X}}_{\nS}=
\upsigma(x_{1,\lS},\boldsymbol{v}_{\lS})_{0\leq\lS\leq\nS}$. In 
addition, assume that the following hold:
\begin{enumerate}
\item
$\sum_{\nnn}\sqrt{\EC{\|{a}_{1,\nS}\|^2}
{\boldsymbol{\EuScript{X}}_{\nS}}}<\pinf$ and
$\sum_{\nnn}\sqrt{\EC{\|\boldsymbol{b}_{\nS}\|^2}
{\boldsymbol{\EuScript{X}}_{\nS}}}<\pinf$.
\item
For every $\nnn$, $\boldsymbol{\EuScript{E}}_{\nS}$
and $\boldsymbol{\EuScript{X}}_{\nS}$ are independent.
\item
For every $\lS\in\{1,\ldots,\mathsf{p}\}$,
$\PP[\varepsilon_{\lS,0}=1]>0$.
\end{enumerate}
Then $(x_{1,\nS})_{\nnn}$ converges weakly $\Pas$ to a $\ZS$-valued
random variable.
\end{proposition}

Let us note that Algorithm \eqref{e:71} satisfies
{\bfseries R1} but not {\bfseries R2} since it must activate
$\mathsf{A}$ at each iteration, nor {\bfseries R3} since it
requires the knowledge of the norms of linear operators to
implement \eqref{e:PRcond}. Another framework related to
Problem~\ref{prob:1} is that of \cite{Davi23}, which allows for
random activations in Problem~\ref{prob:1} in a finite-dimensional
setting when no linear operator is present and under the
assumption that the operators
$(\mathsf{B}_{\kS})_{1\leq\kS\leq\mathsf{p}}$ are cocoercive. It
therefore does not satisfy several requirements of {\bfseries R1}
and activates only one operator per iteration, which violates
{\bfseries R2}. On the other hand, the recent work
\cite{Sadi24} solves Problem~\ref{prob:1} in a 
finite-dimensional setting when no linear operator is present and
under strong monotonicity of the nonlinear operators. Hence, it
does not satisfy {\bfseries R1} and, since it does not allow for 
multiple activations at each iteration, it does not satisfy
{\bfseries R2} either.

\section{Proposed algorithms}
\label{sec:3}

\subsection{Multivariate framework}

Our strategy consists in embedding Problem~\ref{prob:1} into
multivariate problems that have the following general form studied
in \cite{Siop15} and involve $\mathsf{m}$ agents 
$(\mathsf{x}_{1},\ldots,\mathsf{x}_{\mathsf{m}})$.

\begin{problem}
\label{prob:2}
Let $(\XS_{\iS})_{1\leq\iS\leq\mathsf{m}}$ and
$(\YS_{\jS})_{1\leq\jS\leq\mathsf{r}}$ be families of separable 
real Hilbert spaces with Hilbert direct sums 
$\XXS=\XS_1\oplus\cdots\oplus\XS_{\mathsf{m}}$ and
$\YYS=\YS_1\oplus\cdots\oplus\YS_{\mathsf{r}}$. For every 
$\iS\in\{1,\ldots,\mathsf{m}\}$ and every
$\jS\in\{1,\ldots,\mathsf{r}\}$, let 
$\mathsf{C}_{\iS}\colon\XS_{\iS}\to 2^{\XS_{\iS}}$ and
$\mathsf{D}_{\jS}\colon\YS_{\jS}\to 2^{\YS_{\jS}}$ be maximally
monotone, and let $\mathsf{M}_{\jS\iS}\colon\XS_{\iS}\to\YS_{\jS}$ 
be linear and bounded. Set
\begin{equation}
\label{e:A}
\begin{cases}
\boldsymbol{\mathsf{M}}\colon\XXS\to\YYS\colon
\boldsymbol{\mathsf{x}}\mapsto\brk1{
\sum_{\iS=1}^{\mathsf{m}}\mathsf{M}_{1\iS}\mathsf{x}_{\iS},
\ldots,
\sum_{\iS=1}^{\mathsf{m}}\mathsf{M}_{\mathsf{r}\iS}
\mathsf{x}_{\iS}}\\
\boldsymbol{\mathsf{C}}\colon\XXS\to 2^{\XXS}\colon
\boldsymbol{\mathsf{x}}\mapsto
\mathsf{C}_1\mathsf{x}_1\times\cdots\times
\mathsf{C}_{\mathsf{m}}\mathsf{x}_{\mathsf{m}}\\
\boldsymbol{\mathsf{D}}\colon\YYS\to 2^{\YYS}\colon
\boldsymbol{\mathsf{y}}\mapsto
\mathsf{D}_1\mathsf{y}_1\times\cdots\times
\mathsf{D}_{\mathsf{r}}\mathsf{y}_{\mathsf{r}}.
\end{cases}
\end{equation}
The task is to
\begin{equation}
\label{e:p2}
\text{find}\;\;\boldsymbol{\mathsf{x}}\in\XXS
\;\;\text{such that}\;\;
\boldsymbol{\mathsf{0}}\in
\boldsymbol{\mathsf{C}}\boldsymbol{\mathsf{x}}
+\boldsymbol{\mathsf{M}}^*\brk1{\boldsymbol{\mathsf{D}}
(\boldsymbol{\mathsf{M}}\boldsymbol{\mathsf{x}})}.
\end{equation}
The set of solutions to \eqref{e:p2} is denoted by
$\boldsymbol{\mathsf{Z}}$ and assumed to be nonempty. Further, the
projection operator onto the subspace 
\begin{equation}
\label{e:2014}
\boldsymbol{\mathsf{V}}=
\menge{(\boldsymbol{\mathsf{x}},\boldsymbol{\mathsf{y}})
\in\XXS\oplus\YYS}{\boldsymbol{\mathsf{y}}=\boldsymbol{\mathsf{M}}
\boldsymbol{\mathsf{x}}}
\end{equation}
is decomposed as 
\begin{equation}
\label{e:hu}
\proj_{\boldsymbol{\mathsf{V}}}
\colon(\boldsymbol{\mathsf{x}},\boldsymbol{\mathsf{y}})\mapsto
\bigl(\mathsf{Q}_{\lS}(\boldsymbol{\mathsf{x}},
\boldsymbol{\mathsf{y}})\bigr)_{1\leq\lS\leq
\mathsf{m}+\mathsf{r}},\quad\text{where}\quad
\begin{cases}
(\forall\iS\in\{1,\ldots,\mathsf{m}\})\;\;
\mathsf{Q}_{\iS}\colon\XXS\oplus\YYS\to\XS_{\iS}\\
(\forall\jS\in\{1,\ldots,\mathsf{r}\})\;\;
\mathsf{Q}_{\mathsf{m}+\jS}\colon\XXS\oplus\YYS\to\YS_{\jS}.
\end{cases}
\end{equation}
\end{problem}

Our approach is ultimately based on the Douglas--Rachford
algorithm implemented in $\XXS\oplus\YYS$. Define 
\begin{equation}
\boldsymbol{\mathsf{A}}\colon\XXS\oplus\YYS\to 2^{\XXS\oplus\YYS}
\colon(\boldsymbol{\mathsf{x}},\boldsymbol{\mathsf{y}})\mapsto
\boldsymbol{\mathsf{C}}\boldsymbol{\mathsf{x}}\times
\boldsymbol{\mathsf{D}}\boldsymbol{\mathsf{y}}
\quad\text{and}\quad\boldsymbol{\mathsf{B}}=
\boldsymbol{\mathsf{N}}_{\boldsymbol{\mathsf{V}}}.
\end{equation}
Then it follows from \cite[Eq.~(5.23)]{Siop15} that 
$(\boldsymbol{\mathsf{x}},\boldsymbol{\mathsf{y}})
\in\zer(\boldsymbol{\mathsf{A}}+\boldsymbol{\mathsf{B}})$ if and
only if 
$\boldsymbol{\mathsf{x}}\in\zer(\boldsymbol{\mathsf{C}}+
\boldsymbol{\mathsf{M}}^*\circ\boldsymbol{\mathsf{D}}\circ
\boldsymbol{\mathsf{M}})$ and $\boldsymbol{\mathsf{y}}
=\boldsymbol{\mathsf{M}}\boldsymbol{\mathsf{x}}$.
We can construct a point in 
$\zer(\boldsymbol{\mathsf{A}}+\boldsymbol{\mathsf{B}})$
iteratively by the Douglas--Rachford algorithm
\cite[Section~26.3]{Livre1}, which requires the resolvents of 
$\boldsymbol{\mathsf{A}}$ and $\boldsymbol{\mathsf{B}}$. By 
\cite[Proposition~23.18]{Livre1}, 
$\boldsymbol{\mathsf{J}}_{\boldsymbol{\mathsf{A}}}$
can be decomposed in terms of $(\mathsf{J}_{\mathsf{C}_1},\ldots,
\mathsf{J}_{\mathsf{C}_{\mathsf{m}}},
\mathsf{J}_{\mathsf{D}_1},\ldots,
\mathsf{J}_{\mathsf{D}_{\mathsf{r}}})$. On the other hand, 
$\boldsymbol{\mathsf{J}}_{\boldsymbol{\mathsf{B}}}=
\proj_{\boldsymbol{\mathsf{V}}}$ and it
follows from \eqref{e:2014} and \cite[Example~29.19(ii)]{Livre1} 
that
\begin{equation}
\label{e:r3}
(\forall\boldsymbol{\mathsf{x}}\in\XXS)
(\forall\boldsymbol{\mathsf{y}}\in\YYS)\quad
\proj_{\boldsymbol{\mathsf{V}}}
(\boldsymbol{\mathsf{x}},\boldsymbol{\mathsf{y}})
=(\boldsymbol{\mathsf{p}},\boldsymbol{\mathsf{M}}
\boldsymbol{\mathsf{p}}),\quad\text{where}\quad
\boldsymbol{\mathsf{p}}=
(\ID+\boldsymbol{\mathsf{M}}^*\circ\boldsymbol{\mathsf{M}})^{-1}
\bigl(\boldsymbol{\mathsf{x}}+
\boldsymbol{\mathsf{M}}^*\boldsymbol{\mathsf{y}}\bigr).
\end{equation}
This operator is decomposed in terms of the operators
$(\mathsf{Q}_{\lS})_{1\leq\lS\leq\mathsf{m}+\mathsf{r}}$ in 
\eqref{e:hu}. The following result provides a randomly
block-activated implementation of this product space version of the
Douglas--Rachford algorithm.

\begin{theorem}[{\cite[Corollary~5.3]{Siop15}}]
\label{t:1}
Consider the setting of Problem~\ref{prob:2}. Set 
$\mathsf{O}=\{0,1\}^{\mathsf{m}+\mathsf{r}}\smallsetminus
\{\boldsymbol{\mathsf{0}}\}$, let $\upgamma\in\RPP$, let 
$(\uplambda_{\nS})_{\nnn}$ be a sequence in $\left]0,2\right[$ such
that $\inf_{\nnn}\uplambda_{\nS}>0$ and 
$\sup_{\nnn}\uplambda_{\nS}<2$, let $\boldsymbol{x}_0$, 
$\boldsymbol{z}_0$, $(\boldsymbol{a}_{\nS})_{\nnn}$, and 
$(\boldsymbol{c}_{\nS})_{\nnn}$
be $\XXS$-valued random variables, let $\boldsymbol{y}_0$,
$\boldsymbol{w}_0$, $(\boldsymbol{b}_{\nS})_{\nnn}$, and 
$(\boldsymbol{d}_{\nS})_{\nnn}$ be $\YYS$-valued random 
variables, and let $(\boldsymbol{\varepsilon}_{\nS})_{\nnn}$ be 
identically distributed $\mathsf{O}$-valued random variables. 
Iterate
\begin{equation}
\label{e:a83}
\begin{array}{l}
\text{for}\;\nS=0,1,\ldots\\
\left\lfloor
\begin{array}{l}
\text{for}\;\iS=1,\ldots,\mathsf{m}\\
\left\lfloor
\begin{array}{l}
x_{\iS,\nS+1}=x_{\iS,\nS}+\varepsilon_{\iS,\nS}
\bigl(\mathsf{Q}_{\iS}(\boldsymbol{z}_{\nS},\boldsymbol{w}_{\nS})
+a_{\iS,\nS}-x_{\iS,\nS}\bigr)\\[1mm]
z_{\iS,\nS+1}=z_{\iS,\nS}+\varepsilon_{\iS,\nS}\uplambda_{\nS}
\bigl(\mathsf{J}_{\upgamma\mathsf{C}_{\iS}}
(2x_{\iS,\nS+1}-z_{\iS,\nS})+c_{\iS,\nS}-x_{\iS,\nS+1}\bigr)
\end{array}
\right.\\
\text{for}\;\jS=1,\ldots,\mathsf{r}\\
\left\lfloor
\begin{array}{l}
y_{\jS,\nS+1}=y_{\jS,\nS}+\varepsilon_{\mathsf{m}+\jS,\nS}
\big(\mathsf{Q}_{\mathsf{m}+\jS}
(\boldsymbol{z}_{\nS},\boldsymbol{w}_{\nS})+b_{\jS,\nS}
-y_{\jS,\nS}\big)\\[1mm]
w_{\jS,\nS+1}=w_{\jS,\nS}
+\varepsilon_{\mathsf{m}+\jS,\nS}\uplambda_{\nS}
\bigl(\mathsf{J}_{\upgamma\mathsf{D}_{\jS}}
(2y_{\jS,\nS+1}-w_{\jS,\nS})+d_{\jS,\nS}-y_{\jS,\nS+1}\bigr),
\end{array}
\right.
\end{array}
\right.\\
\end{array}
\end{equation}
and set $(\forall\nnn)$ $\boldsymbol{\EuScript{E}}_{\nS}=
\upsigma(\boldsymbol{\varepsilon}_{\nS})$
and $\boldsymbol{\EuScript{S}}_{\nS}=\upsigma
(\boldsymbol{z}_{\lS},\boldsymbol{w}_{\lS})_{0\leq{\lS}\leq{\nS}}$.
In addition, assume that the following are satisfied:
\begin{enumerate}
\item
$\sum_{\nnn}\sqrt{\EC{\|\boldsymbol{a}_{\nS}\|^2}
{\boldsymbol{\EuScript{S}}_{\nS}}}<\pinf$,
$\sum_{\nnn}\sqrt{\EC{\|\boldsymbol{b}_{\nS}\|^2}
{\boldsymbol{\EuScript{S}}_{\nS}}}<\pinf$,
$\sum_{\nnn}\sqrt{\EC{\|\boldsymbol{c}_{\nS}\|^2}
{\boldsymbol{\EuScript{S}}_{\nS}}}<\pinf$,
$\sum_{\nnn}\sqrt{\EC{\|\boldsymbol{d}_{\nS}\|^2}
{\boldsymbol{\EuScript{S}}_{\nS}}}<\pinf$, 
$\boldsymbol{a}_{\nS}\weakly\boldsymbol{\mathsf{0}}\:\Pas$, and 
$\boldsymbol{b}_{\nS}\weakly\boldsymbol{\mathsf{0}}\:\Pas$
\item
For every $\nnn$, $\boldsymbol{\EuScript{E}}_{\nS}$ and
$\boldsymbol{\EuScript{S}}_{\nS}$ are independent.
\item
For every $\lS\in\{1,\ldots,\mathsf{m}+\mathsf{r}\}$, 
$\PP[\varepsilon_{\lS,0}=1]>0$.
\end{enumerate}
Then $(\boldsymbol{x}_{\nS})_{\nnn}$ converges weakly $\Pas$ to a
$\boldsymbol{\mathsf{Z}}$-valued random variable.
\end{theorem}

\begin{remark}
\label{r:0}
The measurability of the weak limit in \cite[Corollary~5.3]{Siop15}
relies on \cite[Proposition~2.3]{Siop15}, which involves 
Pettis' theorem \cite[Corollary~1.13]{Pett38}. The applicability 
of the latter follows from the separability of $\HS$ and the fact
that $(\upOmega,\FE,\PP)$ is a complete probability space; see
\cite[Sections~1.1a--b]{Hyto16} for details. 
\end{remark}

\begin{remark}
\label{r:1}
At iteration $\nS$, the random variables 
$(\varepsilon_{\iS,\nS})_{1\leq\iS\leq\mathsf{m}}$ and 
$(\varepsilon_{\mathsf{m}+\jS,\nS})_{1\leq\jS\leq\mathsf{r}}$
act as switches which control which components are updated, while
the random variables $a_{\iS,\nS}$, $b_{\jS,\nS}$, $c_{\iS,\nS}$
and $d_{\jS,\nS}$ model approximations in the
implementation of the operators $\mathsf{Q}_{\iS}$,
$\mathsf{Q}_{\jS}$, $\mathsf{J}_{\upgamma\mathsf{C}_{\iS}}$, and 
$\mathsf{J}_{\upgamma\mathsf{D}_{\jS}}$, respectively.
\end{remark}

We now present three frameworks for solving Problem~\ref{prob:1}
which are based on specializations of Theorem~\ref{t:1}. 

\subsection{Framework 1}
\label{sec:f1}
The first approach stems from the observation that 
Problem~\ref{prob:2} reduces to Problem~\ref{prob:1} when
$\mathsf{m}=1$, $\mathsf{r}=\mathsf{p}$, $\XS_1=\HS$, 
$\mathsf{C}_1=\mathsf{A}$, and 
$(\forall\kS\in\{1,\ldots,\mathsf{p}\})$ 
$\YS_{\kS}=\GS_{\kS}$,
$\mathsf{M}_{\kS,1}=\LS_{\kS}$, and
$\mathsf{D}_{\kS}=\mathsf{B}_{\kS}$.
Surprisingly, this basic observation does not seem to have been
exploited in attempts to design random block activation algorithms
for solving Problem~\ref{prob:10} or Problem~\ref{prob:1} (or
special cases thereof) using the
stochastic quasi-Fej\'er framework of \cite{Siop15}; see for
instance \cite{Cham24,Mali20,Mimo24,Peng16}. 

We derive from Theorem~\ref{t:1} the following convergence result.

\begin{proposition}
\label{p:3}
Consider the setting of Problem~\ref{prob:1}.
Set $\mathsf{O}=\{0,1\}^{1+\mathsf{p}}\smallsetminus
\{\boldsymbol{\mathsf{0}}\}$, let $\upgamma\in\RPP$, let
$(\uplambda_{\nS})_{\nnn}$ be a sequence in $\left]0,2\right[$ 
such that $\inf_{\nnn}\uplambda_{\nS}>0$ and 
$\sup_{\nnn}\uplambda_{\nS}<2$, let $x_{1,0}$, $z_{1,0}$,
$(c_{1,\nS})_{\nnn}$, and $(e_{\nS})_{\nnn}$ be $\HS$-valued random
variables, let $\boldsymbol{y}_0$, $\boldsymbol{w}_0$, and 
$(\boldsymbol{d}_{\nS})_{\nnn}$ be $\GGS$-valued random
variables, and let $(\boldsymbol{\varepsilon}_\nS)_{\nnn}$ be 
identically distributed $\mathsf{O}$-valued random variables.
Set $\mathsf{Q}=(\Id+\sum_{\kS=1}^{\mathsf{p}}
\LS_{\kS}^*\circ\LS_{\kS})^{-1}$ and iterate 
\begin{equation}
\label{e:algo1}
\begin{array}{l}
\text{for}\;\nS=0,1,\ldots\\
\left\lfloor
\begin{array}{l}
s_{\mathsf{n}}
=\mathsf{Q}\brk1{z_{1,\nS}+\sum_{\kS=1}^{\mathsf{p}}
\LS_{\kS}^*w_{\kS,\nS}}+e_{\nS}\\
x_{1,\nS+1}=
x_{1,\nS}+\varepsilon_{1,\nS}(s_{\nS}-x_{1,\nS})\\
z_{1,\nS+1}=
z_{1,\nS}+\varepsilon_{1,\nS}\uplambda_{\nS}
\brk1{\mathsf{J}_{\upgamma\mathsf{A}}
(2x_{1,\nS+1}-z_{1,\nS})+c_{1,\nS}-x_{1,\nS+1}}\\
\text{for}\;\kS=1,\dots,\mathsf{p}\\
\left\lfloor
\begin{array}{l}
y_{\kS,\nS+1}=
y_{\kS,\nS}+\varepsilon_{1+\kS,\nS}
(\LS_{\kS}s_{\nS}-{y}_{\kS,\nS})\\
w_{\kS,\nS+1}=w_{\kS,\nS}+\varepsilon_{1+\kS,\nS}\uplambda_{\nS}
\brk1{\mathsf{J}_{\upgamma\mathsf{B}_{\kS}}
(2y_{\kS,\nS+1}-w_{\kS,\nS})+d_{\kS,\nS}-y_{\kS,\nS+1}}.
\end{array}
\right.\\
\end{array}
\right.\\
\end{array}
\end{equation}
In addition, assume that the following are satisfied:
\begin{enumerate}
\item
\label{p:3i}
$\sum_{\nnn}\sqrt{\EC{\|c_{1,\nS}\|^2}
{\upsigma({z}_{1,\lS},\boldsymbol{v}_{\lS})_{0\leq 
\lS\leq\nS}}}<\pinf$,
$\sum_{\nnn}\sqrt{\EC{\|\boldsymbol{d}_{\nS}\|^2}
{\upsigma({z}_{1,\lS},
\boldsymbol{v}_{\lS})_{0\leq\lS\leq\nS}}}<\pinf$,\\
$\sum_{\nnn}\sqrt{\EC{\|e_{\nS}\|^2}
{\upsigma({z}_{1,\lS},\boldsymbol{v}_{\lS})_{0\leq 
\lS\leq\nS}}}<\pinf$, and $e_{\nS}\weakly\mathsf{0}$.
\item
\label{p:3ii}
For every $\nnn$, $\upsigma(\boldsymbol{\varepsilon}_{\mathsf{n}})$
and $\upsigma(z_{1,\lS},\boldsymbol{w}_{\lS})_{0\leq 
\lS\leq\nS}$ are independent.
\item
\label{p:3iii}
For every $\lS\in\{1,\ldots,\mathsf{p}+1\}$,
$\PP[\varepsilon_{\lS,0}=1]>0$.
\end{enumerate}
Then $(x_{1,\nS})_{\nnn}$ converges weakly $\Pas$ to a
$\ZS$-valued random variable. 
\end{proposition}
\begin{proof}
In Problem~\ref{prob:2}, set $\mathsf{m}=1$, 
$\mathsf{r}=\mathsf{p}$, $\XS_1=\HS$, $\mathsf{C}_1=\mathsf{A}$, 
and, for every $\kS\in\{1,\ldots,\mathsf{p}\}$, 
$\YS_{\kS}=\GS_{\kS}$, $\mathsf{M}_{\kS,1}=\LS_{\kS}$, and 
$\mathsf{D}_{\kS}=\mathsf{B}_{\kS}$. 
Further, for every $\nnn$, set $a_{1,\nS}=e_{\nS}$ and,
for every $\kS\in\{1,\dots,\mathsf{p}\}$, 
set $b_{\kS,\nS}=\LS_{\kS}e_{\nS}$. Then it follows from \ref{p:3i}
that $a_{1,\nS}\weakly\mathsf{0}$ $\Pas$, 
$\boldsymbol{b}_{\nS}\weakly\boldsymbol{\mathsf{0}}$ $\Pas$, and
\begin{align}
\sum_{\nnn}\sqrt{\EC{\|\boldsymbol{b}_{\nS}\|^2}
{{\upsigma(\boldsymbol{z}_{\lS},\boldsymbol{v}_{\lS})_{0\leq 
\lS\leq\nS}}}}
&\leq\sum_{\nnn}\sqrt{\EC4{\brk3{\sum_{\kS=1}^{\mathsf{p}}
\|\LS_{\kS}\|^2}\|e_{\nS}\|^2}
{{\upsigma(\boldsymbol{z}_{\lS},\boldsymbol{v}_{\lS})_{0\leq 
\lS\leq\nS}}}}\nonumber\\
&=\sqrt{\sum_{\kS=1}^{\mathsf{p}}
\|\LS_{\kS}\|^2}\sum_{\nnn}\sqrt{\EC1{\|e_{\nS}\|^2}
{{\upsigma(\boldsymbol{z}_{\lS},\boldsymbol{v}_{\lS})_{0\leq 
\lS\leq\nS}}}}\nonumber\\
&<\pinf. 
\end{align}
The assertion therefore results from Theorem~\ref{t:1}.
\end{proof}

\subsection{Framework 2}
\label{sec:f2}

In Framework~1, Problem~\ref{prob:2} collapses to
Problem~\ref{prob:1} by reducing the number of agents to
$\mathsf{m}=1$. Here, we use $\mathsf{m}=\mathsf{p}+1$ agents in
Problem~\ref{prob:2} and capture Problem~\ref{prob:1} by forcing
these agents $(\mathsf{x}_{1},\ldots,\mathsf{x}_{\mathsf{p}+1})$ to
lie in the subspace $\boldsymbol{\mathsf{W}}$ of \eqref{e:2024}.

\begin{proposition}
\label{p:21}
Consider the setting of Problem~\ref{prob:1}.
Set $\mathsf{O}=\{0,1\}^{\mathsf{p}+2}\smallsetminus
\{\boldsymbol{\mathsf{0}}\}$, let 
$\upgamma\in\RPP$, let $(\uplambda_{\nS})_{\nnn}$ be a sequence in
$\left]0,2\right[$ such that $\inf_{\nnn}\uplambda_{\nS}>0$ and 
$\sup_{\nnn}\uplambda_{\nS}<2$, let $\boldsymbol{x}_0$, 
$\boldsymbol{z}_0$, $\boldsymbol{u}_0$, $\boldsymbol{v}_0$, and
$(\boldsymbol{c}_{\nS})_{\nnn}$ be 
$\HS\oplus\GGS$-valued random
variables, let $(e_{\nS})_{\nnn}$ be $\HS$-valued random
variables, and let $(\boldsymbol{\varepsilon}_\nS)_{\nnn}$
be identically distributed $\mathsf{O}$-valued random variables. 
Set $\mathsf{Q}=(\Id+\sum_{\kS=1}^{\mathsf{p}}
\LS_{\kS}^*\circ\LS_{\kS})^{-1}$. Iterate 
\begin{equation}
\label{e:algo2}
\begin{array}{l}
\text{for}\;\nS=0,1,\ldots\\
\left\lfloor
\begin{array}{l}
\text{for}\;\iS=1,\dots,\mathsf{p}+1\\
\left\lfloor
\begin{array}{l}
x_{\iS,\nS+1}=x_{\iS,\nS}+\varepsilon_{\iS,\nS}\Bigl(
\dfrac{z_{\iS,\nS}+v_{\iS,\nS}}{2}-x_{\iS,\nS}\Bigr)\\[3mm]
\end{array}
\right.\\%[7mm]
z_{1,\nS+1}=z_{1,\nS}+\varepsilon_{1,\nS}\uplambda_{\nS}
\bigl(\mathsf{J}_{\upgamma\mathsf{A}}
(2x_{1,\nS+1}-z_{1,\nS})+c_{1,\nS}-x_{1,\nS+1}\bigr)\\
\text{for}\;\kS=1,\dots,\mathsf{p}\\
\left\lfloor
\begin{array}{l}
z_{\kS+1,\nS+1}=z_{\kS+1,\nS}
+\varepsilon_{\kS+1,\nS}\uplambda_{\nS}
\bigl(\mathsf{J}_{\upgamma\mathsf{B}_{\kS}}
(2x_{\kS+1,\nS+1}-z_{\kS+1,\nS})+c_{\kS+1,\nS}
-x_{\kS+1,\nS+1}\bigr)\\
\end{array}
\right.\\%[7mm]
\text{for}\;\kS=1,\dots,\mathsf{p}+1\\
\left\lfloor
\begin{array}{l}
u_{\kS,\nS+1} 
=u_{\kS,\nS}+\varepsilon_{\mathsf{p}+2,\nS}
\Bigl(\dfrac{{z}_{\kS,\nS}+v_{\kS,\nS}}{2}-
u_{\kS,\nS}\Bigr)\\[3mm]
\end{array}
\right.\\%[7mm]
s_{\nS}
=\varepsilon_{\mathsf{p}+2,\nS}\brk2{\mathsf{Q}
\bigl(2u_{1,\nS+1}-v_{1,\nS}+\sum_{\kS=1}^{\mathsf{p}}
\LS_{\kS}^*(2u_{\kS+1,\nS+1}-v_{\kS+1,\nS})\bigr)+e_{\nS}}\\
v_{1,\mathsf{n}+1}
=v_{1,\mathsf{n}}+\varepsilon_{\mathsf{p}+2,\nS}\uplambda_{\nS}
(s_{\nS}-u_{1,\nS+1})\\
\text{for}\;\kS=1,\dots,\mathsf{p}\\
\left\lfloor
\begin{array}{l}
v_{\kS+1,\nS+1}
=v_{\kS+1,\nS}+\varepsilon_{\mathsf{p}+2,\nS}\uplambda_{\nS}
(\LS_{\kS}s_{\nS}-u_{\kS+1,\nS+1}).
\end{array}
\right.\\
\end{array}
\right.\\
\end{array}
\end{equation}
In addition, assume that the following are satisfied:
\begin{enumerate}
\item
\label{p:65i}
$\sum_{\nnn}
\sqrt{\EC{\|\boldsymbol{c}_{\nS}\|^2}
{\upsigma(\boldsymbol{z}_{\lS},
\boldsymbol{v}_{\lS})_{0\leq\lS\leq\nS}}}<\pinf$ and 
$\sum_{\nnn}\sqrt{\EC{\|e_{\nS}\|^2}
{\upsigma(\boldsymbol{z}_{\lS},\boldsymbol{v}_{\lS})_{0\leq 
\lS\leq\nS}}}<\pinf$.
\item
\label{p:65ii}
For every $\nnn$, $\upsigma(\boldsymbol{\varepsilon}_{\mathsf{n}})$
and $\upsigma(\boldsymbol{z}_{\lS},\boldsymbol{v}_{\lS})_{0\leq 
\lS\leq\nS}$ are independent.
\item
\label{p:65iii}
For every $\lS\in\{1,\ldots,\mathsf{p}+2\}$,
$\PP[\varepsilon_{\lS,0}=1]>0$.
\end{enumerate}
Then $(x_{1,\nS})_{\nnn}$ converges weakly $\Pas$ to a $\ZS$-valued
random variable. 
\end{proposition}
\begin{proof}
In Problem~\ref{prob:2}, set $\mathsf{m}=\mathsf{p}+1$,
$\mathsf{r}=1$, $\XS_1=\HS$, $(\XS_{\iS})_{2\leq\iS\leq\mathsf{m}}=
(\GS_{\iS-1})_{2\leq\iS\leq\mathsf{m}}$, and $\YS_1=\HS\oplus\GGS$.
Thus, $\YYS=\YS_1=\HS\oplus\GGS=\XXS$. Moreover,
for every $\iS\in\{1,\ldots,\mathsf{p}+1\}$, set
\begin{equation}
\mathsf{M}_{1\iS}\colon\XS_{\iS}\to\HS\oplus\GGS
\colon\mathsf{x}_{\iS}
\mapsto\brk1{\mathsf{0},\ldots,\mathsf{0},
\underbrace{\mathsf{x}_{\iS}}_{\iS\text{th position}},
\mathsf{0},\ldots,\mathsf{0}}, 
\end{equation}
which yields
\begin{equation}
\label{e:fa}
\mathsf{M}_{1\iS}^*\colon\HS\oplus\GGS\to\XS_{\iS}\colon
(\mathsf{x}_1^*,\ldots,\mathsf{x}_{\mathsf{p}+1}^*)
\mapsto\mathsf{x}^*_{\iS}.
\end{equation}
Further, denote by 
$\boldsymbol{\mathsf{x}}=(\mathsf{x}_{1},\dots,
\mathsf{x}_{\mathsf{p}+1})$ a generic element in $\HS\oplus\GGS$
and define
$\mathsf{D}_1=\mathsf{N}_{\boldsymbol{\mathsf{W}}}$, where
$\boldsymbol{\mathsf{W}}$ is the subspace of \eqref{e:2024}.
In this configuration, \eqref{e:p2} reduces to 
\begin{equation}
\label{e:p21}
\text{find}\;\;\boldsymbol{\mathsf{x}}\in\HS\oplus\GGS\;\;
\text{such that}\;\;\boldsymbol{\mathsf{0}}\in
\overset{\mathsf{p}+1}{\underset{\iS=1}{\bigtimes}}\;
\mathsf{C}_{\iS}\mathsf{x}_{\mathsf{i}}
+\mathsf{N}_{\boldsymbol{\mathsf{W}}}\boldsymbol{\mathsf{x}}.
\end{equation} 
We observe that 
\begin{equation}
(\forall\iS\in\{1,\ldots,\mathsf{p}+1\})
(\forall\lS\in\{1,\ldots,\mathsf{p}+1\})\quad
\mathsf{M}_{1\iS}^*\circ\mathsf{M}_{1\lS}=
\begin{cases}
\Id,&\text{if}\;\;\iS=\lS;\\
\mathsf{0},&\text{if}\;\;\iS\neq\lS.
\end{cases}
\end{equation}
As a result, 
$(\ID+\boldsymbol{\mathsf{M}}^*\circ
\boldsymbol{\mathsf{M}})^{-1}=(1/2)\ID$ and we derive from
\eqref{e:hu}, \eqref{e:r3}, and \eqref{e:fa} that
\begin{equation}
\label{e:jpf2}
\mathsf{Q}_{\mathsf{p}+2}\colon(\boldsymbol{z},\boldsymbol{v})
\mapsto\dfrac{\boldsymbol{z}+\boldsymbol{v}}{2}
\quad\text{and}\quad
(\forall\lS\in\{1,\ldots,\mathsf{p}+1\})\quad
\mathsf{Q}_{\lS}\colon(\boldsymbol{z},\boldsymbol{v})
\mapsto\dfrac{z_{\lS}+v_{\lS}}{2}.
\end{equation}
Altogether, \eqref{e:a83} with variables
$y_{1,\nS}=\boldsymbol{u}_{\nS}\in\HS\oplus\GGS$ and
$w_{1,\nS}=\boldsymbol{v}_{\nS}\in\HS\oplus\GGS$ becomes
\begin{equation}
\label{e:a89}
\begin{array}{l}
\text{for}\;\nS=0,1,\ldots\\
\left\lfloor
\begin{array}{l}
\text{for}\;\iS=1,\ldots,\mathsf{p}+1\\
\left\lfloor
\begin{array}{l}
x_{\iS,\nS+1}=x_{\iS,\nS}+\varepsilon_{\iS,\nS}
\Bigl(\dfrac{z_{\iS,\nS}+v_{\iS,\nS}}{2}
-x_{\iS,\nS}\Bigr)\\[2mm]
z_{\iS,\nS+1}=z_{\iS,\nS}+\varepsilon_{\iS,\nS}\uplambda_{\nS}
\bigl(\mathsf{J}_{\upgamma\mathsf{C}_{\iS}}
(2x_{\iS,\nS+1}-z_{\iS,\nS})+c_{\iS,\nS}-
x_{\iS,\nS+1}\bigr)
\end{array}
\right.\\[5mm]
\boldsymbol{u}_{\nS+1} 
=\boldsymbol{u}_{\nS}+\varepsilon_{\mathsf{p}+2,\nS}
\Bigl(\dfrac{\boldsymbol{z}_{\nS}+\boldsymbol{v}_{\nS}}{2}
-\boldsymbol{u}_{\nS}\Bigr)\\[2mm]
\boldsymbol{v}_{\nS+1}=\boldsymbol{v}_{\nS}
+\varepsilon_{\mathsf{p}+2,\nS}\uplambda_{\nS}
\bigl(\proj_{\boldsymbol{\mathsf{W}}}
(2\boldsymbol{u}_{\nS+1}-\boldsymbol{v}_{\nS})+
\boldsymbol{d}_{1,\nS}-\boldsymbol{u}_{\nS+1}\bigr),
\end{array}
\right.\\
\end{array}
\end{equation}
where $\boldsymbol{d}_{1,\nS}$ is the error incurred when
projecting onto $\boldsymbol{\mathsf{W}}$ at iteration $\nS$.
We derive from \eqref{e:2024} and 
\cite[Example~29.19(ii)]{Livre1} that
\begin{multline}
\proj_{\boldsymbol{\mathsf{W}}}\colon
(\mathsf{x},\mathsf{y}_1,\ldots,\mathsf{y}_{\mathsf{p}})\mapsto
(\mathsf{s},\LS_1\mathsf{s},\ldots,
\LS_{\mathsf{p}}\mathsf{s}),\\
\text{where}\quad\mathsf{s}=
\brk3{\Id+\sum_{\kS=1}^{\mathsf{p}}\LS_{\kS}^*\circ\LS_{\kS}}^{-1}
\brk3{\mathsf{x}+\sum_{\kS=1}^{\mathsf{p}}\LS_{\kS}^*
\mathsf{y}_{\kS}}.
\end{multline}
Set $(\forall\nnn)$ 
$\boldsymbol{d}_{1,\nS}=(e_{\nS},\LS_1e_{\nS},\ldots,
\LS_{\mathsf{p}}e_{\nS})$. Then we infer from \ref{p:65i} that
\begin{align}
\sum_{\nnn}\sqrt{\EC1{\|\boldsymbol{d}_{1,\nS}\|^2}
{{\upsigma(\boldsymbol{z}_{\lS},\boldsymbol{v}_{\lS})_{0\leq 
\lS\leq\nS}}}}
&\leq\sum_{\nnn}\sqrt{\EC4{\brk3{1+\sum_{\kS=1}^{\mathsf{p}}
\|\LS_{\kS}\|^2}\|e_{\nS}\|^2}
{{\upsigma(\boldsymbol{z}_{\lS},\boldsymbol{v}_{\lS})_{0\leq 
\lS\leq\nS}}}}\nonumber\\
&=\sqrt{1+\sum_{\kS=1}^{\mathsf{p}}
\|\LS_{\kS}\|^2}\sum_{\nnn}\sqrt{\EC1{\|e_{\nS}\|^2}
{{\upsigma(\boldsymbol{z}_{\lS},\boldsymbol{v}_{\lS})_{0\leq 
\lS\leq\nS}}}}\nonumber\\
&<\pinf. 
\end{align}
Thus, it follows from Theorem~\ref{t:1} that, with 
$\boldsymbol{\mathsf{Z}}$ denoting the set of solutions to 
\eqref{e:p21}, 
\begin{multline}
\label{e:n1983}
\text{the sequence}\;
\brk1{x_{1,\nS},x_{2,\nS},\ldots,x_{\mathsf{p}+1,\nS}}_{\nnn}
\;\text{in \eqref{e:a89}}\;
\text{converges weakly}\;\Pas\\ \text{to a}\;\boldsymbol{\ZS}
\text{-valued random variable}
\;\overline{\boldsymbol{x}}=
\brk1{\overline{x}_1,\overline{x}_2,\ldots,
\overline{x}_{\mathsf{p}+1}}
\:\;\text{if}\:\;\boldsymbol{\mathsf{Z}}\neq\emp.
\end{multline}
Next, we specialize \eqref{e:p21} to 
\begin{equation}
\label{e:v1983}
\mathsf{C}_{1}=\mathsf{A}\quad\text{and}\quad
(\forall\iS\in\{2,\ldots,\mathsf{p}+1\})\quad
\mathsf{C}_{\iS}=\mathsf{B}_{\iS-1}. 
\end{equation}
In this context, \eqref{e:a89} reduces to \eqref{e:algo2}.
Recalling that $\mathsf{Z}$ denotes the set of solutions to
Problem~\ref{prob:1}, in view of \eqref{e:n1983}, it remains 
to show that
\begin{equation}
\boldsymbol{\mathsf{Z}}=\menge{
(\mathsf{x}_1,\LS_1\mathsf{x}_1,\dots,
\LS_{\mathsf{p}}\mathsf{x}_1)}{\mathsf{x}_1\in\mathsf{Z}}.
\end{equation}
Let ${\boldsymbol{\mathsf{x}}}\in\HS\oplus\GGS$. We have
\begin{align}
\boldsymbol{\mathsf{x}}\in
\boldsymbol{\mathsf{Z}}
&\Leftrightarrow
\boldsymbol{\mathsf{x}}\;\text{solves \eqref{e:p21}}
\nonumber\\
&\Leftrightarrow\;
\begin{cases}
{\boldsymbol{\mathsf{x}}}\in\boldsymbol{\mathsf{W}}\\
(\exi\boldsymbol{\mathsf{x}}^*\in\boldsymbol{\mathsf{W}}^\bot)
\quad\boldsymbol{\mathsf{0}}\in{\overset{\mathsf{p}+1}
{\underset{\iS=1}{\times}}} 
\mathsf{C}_{\iS}{\mathsf{x}}_{\iS}
+\boldsymbol{\mathsf{x}}^*
\end{cases}
\nonumber\\
&\Leftrightarrow\;
\begin{cases}
(\exi\mathsf{x}_1\in\HS)\quad
{\boldsymbol{\mathsf{x}}}=
(\mathsf{x}_1,\LS_1\mathsf{x}_1,\dots,
\LS_{\mathsf{p}}\mathsf{x}_1)\\
(\exi\boldsymbol{\mathsf{x}}^*\in\boldsymbol{\mathsf{W}}^\bot)
\quad\boldsymbol{\mathsf{0}}\in
\mathsf{A{x}}_1\times
\mathsf{B}_1(\LS_1{\mathsf{x}}_1)\times\cdots\times
\mathsf{B}_{\mathsf{p}}(\LS_{\mathsf{p}}
{\mathsf{x}}_1)+\boldsymbol{\mathsf{x}}^*
\end{cases}
\nonumber\\
&\Leftrightarrow\;
\begin{cases}
(\exi\mathsf{x}_1\in\HS)\quad
{\boldsymbol{\mathsf{x}}}=
(\mathsf{x}_1,\LS_1\mathsf{x}_1,\dots,
\LS_{\mathsf{p}}\mathsf{x}_1)\\
(\exi(\mathsf{y}_1^*,\dots,\mathsf{y}_{\mathsf{p}}^*)\in\GGS)\quad
(\mathsf{0},\mathsf{0},\ldots,\mathsf{0})\in\\
\hspace{22mm}\mathsf{A{x}}_1\times
\mathsf{B}_1(\LS_1{\mathsf{x}}_1)\times\cdots\times
\mathsf{B}_{\mathsf{p}}(\LS_{\mathsf{p}}
{\mathsf{x}}_1)+
\biggl(\Sum_{\kS=1}^{\mathsf{p}}
\LS_{\kS}^*\mathsf{y}_{\kS}^*,-\mathsf{y}_{1}^*,\dots,
-\mathsf{y}_{\mathsf{p}}^*\biggr)
\end{cases}
\nonumber\\
&\Leftrightarrow\;
\begin{cases}
(\exi\mathsf{x}_1\in\HS)\quad
{\boldsymbol{\mathsf{x}}}=
(\mathsf{x}_1,\LS_1\mathsf{x}_1,\dots,
\LS_{\mathsf{p}}\mathsf{x}_1)\\
(\exi(\mathsf{y}_1^*,\dots,\mathsf{y}_{\mathsf{p}}^*)\in\GGS)\quad
\begin{cases}
\mathsf{0}\in\mathsf{A{x}}_1+\sum_{\kS=1}^{\mathsf{p}}
\LS_{\kS}^*\mathsf{y}_{\kS}^*\\
(\forall\kS\in\{1,\ldots,\mathsf{p}\})\quad
\mathsf{y}_{\kS}^*\in\mathsf{B}_{\kS}(\LS_{\kS}
\mathsf{{\mathsf{x}}_1}).
\end{cases}
\end{cases}
\nonumber\\
&\Leftrightarrow\;
\begin{cases}
(\exi\mathsf{x}_1\in\HS)\quad
{\boldsymbol{\mathsf{x}}}=
(\mathsf{x}_1,\LS_1\mathsf{x}_1,\dots,
\LS_{\mathsf{p}}\mathsf{x}_1)\\
\mathsf{0}\in\mathsf{A}{\mathsf{x}}_1+
\sum_{\kS=1}^\mathsf{p}\LS_{\kS}^*\bigl(\mathsf{B}_{\kS}
(\LS_{\kS}{\mathsf{x}}_1)\bigr)
\end{cases}
\nonumber\\
&\Leftrightarrow\;
(\exi\mathsf{x}_1\in\mathsf{Z})\quad
{\boldsymbol{\mathsf{x}}}=
(\mathsf{x}_1,\LS_1\mathsf{x}_1,\dots,
\LS_{\mathsf{p}}\mathsf{x}_1),
\label{e:999}
\end{align}
which completes the proof. 
\end{proof}

\subsection{Framework 3}
\label{sec:f3}

Our last algorithm connects Problem~\ref{prob:1} to
Problem~\ref{prob:2} by means of a coupling operator
$\boldsymbol{\mathsf{E}}$ mapping to an auxiliary space
$\boldsymbol{\mathsf{K}}$ and such that 
$\ker\boldsymbol{\mathsf{E}}$ coincides with the space
$\boldsymbol{\mathsf{W}}$ of \eqref{e:2024}.

\begin{proposition}
\label{p:2}
Consider the setting of Problem~\ref{prob:1},
let $(\mathsf{K}_{\jS})_{1\leq\jS\leq\mathsf{r}}$ be 
separable real Hilbert spaces, set
\begin{equation}
\label{e:66}
\boldsymbol{\mathsf{K}}=\bigoplus_{\jS=1}^{\mathsf{r}}
\mathsf{K}_{\jS},
\end{equation}
and let 
\begin{equation}
\label{e:67}
\boldsymbol{\mathsf{E}}
\colon\HS\oplus\GGS\to
\boldsymbol{\mathsf{K}}\colon\boldsymbol{\mathsf{x}}
\mapsto\Biggl(\sum_{\iS=1}^{\mathsf{p}+1}\mathsf{E}_{\jS\iS}
{\mathsf{x}_{\iS}}\Biggr)_{1\leq\jS\leq\mathsf{r}}
\end{equation}
be linear, bounded, and such that 
$\ker\boldsymbol{\mathsf{E}}=\boldsymbol{\mathsf{W}}$. Define 
$\boldsymbol{\mathsf{V}}$ as in \eqref{e:2014}, where 
$\XXS$ is replaced with $\HS\oplus\GGS$, 
$\YYS$ with $\boldsymbol{\mathsf{K}}$, 
and $\boldsymbol{\mathsf{M}}$ with 
$\boldsymbol{\mathsf{E}}$, 
and decompose its projection
operator as $\proj_{\boldsymbol{\mathsf{V}}}\colon
\boldsymbol{\mathsf{x}}\mapsto
(\mathsf{R}_{\jS}\boldsymbol{\mathsf{x}})_{1\leq\jS\leq
\mathsf{p}+1+\mathsf{r}}$, 
where $\mathsf{R}_{1}\colon\HS\oplus\GGS\oplus\KKS\to\HS$, 
$(\forall\iS\in\{1,\ldots,\mathsf{p}\})$ 
$\mathsf{R}_{1+\iS}\colon\HS\oplus\GGS\oplus\KKS\to\GS_{\iS}$, 
and $(\forall\kS\in\{1,\ldots,\mathsf{r}\})$ 
$\mathsf{R}_{\mathsf{p}+1+\kS}\colon\HS\oplus\GGS\oplus\KKS\to
\mathsf{K}_{\kS}$. 
Set $\mathsf{O}=\{0,1\}^{\mathsf{p}+1+\mathsf{r}}\smallsetminus
\{\boldsymbol{\mathsf{0}}\}$, let $\upgamma\in\RPP$, let 
$(\uplambda_{\nS})_{\nnn}$ be a sequence in $\left]0,2\right[$ such
that $\inf_{\nnn}\uplambda_{\nS}>0$ and 
$\sup_{\nnn}\uplambda_{\nS}<2$, let $\boldsymbol{x}_0$, 
$\boldsymbol{z}_0$, $(\boldsymbol{a}_{\nS})_{\nnn}$, and 
$(\boldsymbol{c}_{\nS})_{\nnn}$ be 
$\HS\oplus\GGS$-valued random 
variables, let $\boldsymbol{y}_0$, $\boldsymbol{w}_0$, and
$(\boldsymbol{b}_{\nS})_{\nnn}$ be 
$\boldsymbol{\mathsf{K}}$-valued random variables, and let 
$(\boldsymbol{\varepsilon}_{\nS})_{\nnn}$ be identically 
distributed $\mathsf{O}$-valued random variables. Iterate
\begin{equation}
\label{e:algo3}
\begin{array}{l}
\text{for}\;\nS=0,1,\ldots\\
\left\lfloor
\begin{array}{l}
{x}_{1,\nS+1}=
{x}_{1,\nS}+\varepsilon_{1,\nS}
\bigl(\mathsf{R}_{1}
(\boldsymbol{z}_{\nS},\boldsymbol{w}_{\nS})+a_{1,\nS}-x_{1,\nS}
\bigr)\\
z_{1,\nS+1}=
z_{1,\nS}+\varepsilon_{1,\nS}\uplambda_{\nS}
\bigl(\mathsf{J}_{\upgamma \mathsf{A}}
\left(2x_{1,\nS+1}-z_{1,\nS}\right)+c_{1,\nS}-x_{1,\nS+1}\bigr)\\
\text{for}\;\kS=1,\dots,\mathsf{p}\\
\left\lfloor
\begin{array}{l}
{x}_{\kS+1,\nS+1}=
{x}_{\kS+1,\nS}+\varepsilon_{\kS+1,\nS}
\bigl(\mathsf{R}_{\kS+1}
(\boldsymbol{z}_{\nS},\boldsymbol{w}_{\nS})+a_{\kS+1,\nS}-
x_{\kS+1,\nS}\bigr)\\
z_{\kS+1,\nS+1}=
z_{\kS+1,\nS}+\varepsilon_{\kS+1,\nS}\uplambda_{\nS}
\bigl(\mathsf{J}_{\upgamma\mathsf{B}_{\kS}}
\left(2x_{\kS+1,\nS+1}-z_{\kS+1,\nS}\right)+c_{\kS+1,\nS}-
x_{\kS+1,\nS+1}\bigr)\\
\end{array}
\right.\\
\text{for}\;\jS=1,\dots,\mathsf{r}\\
\left\lfloor
\begin{array}{l}
{y}_{\jS,\nS+1}=
{y}_{\jS,\nS}+\varepsilon_{\mathsf{p}+1+\jS,\nS}
\bigl(\mathsf{R}_{\mathsf{p}+1+\jS}(\boldsymbol{z}_{\nS},
\boldsymbol{w}_{\nS})+b_{\jS,\nS}-y_{\jS,\nS}\bigr)\\
{w}_{\jS,\nS+1}=
{w}_{\jS,\nS}-\varepsilon_{\mathsf{p}+1+\jS,\nS}
\uplambda_{\nS}y_{\jS,\nS+1}.
\end{array}
\right.\\
\end{array}
\right.\\
\end{array}
\end{equation}
In addition, assume that the following are satisfied:
\begin{enumerate}
\item
\label{p:2i}
$\sum_{\nnn}\sqrt{\EC{\|\boldsymbol{a}_{\nS}\|^2}
{\upsigma(\boldsymbol{z}_{\lS},\boldsymbol{w}_{\lS})_{0\leq 
\lS\leq\nS}}}<\pinf$,
$\sum_{\nnn}\sqrt{\EC{\|\boldsymbol{b}_{\nS}\|^2}
{\upsigma(\boldsymbol{z}_{\lS},\boldsymbol{w}_{\lS})_{0\leq 
\lS\leq\nS}}}<\pinf$,\\
$\sum_{\nnn}\sqrt{\EC{\|\boldsymbol{c}_{\nS}\|^2}
{\upsigma(\boldsymbol{z}_{\lS},\boldsymbol{w}_{\lS})_{0\leq 
\lS\leq\nS}}}<\pinf$, 
$\boldsymbol{a}_{\nS}\weakly\boldsymbol{\mathsf{0}}\:\Pas$, and 
$\boldsymbol{b}_{\nS}\weakly\boldsymbol{\mathsf{0}}\:\Pas$
\item
\label{p:2ii}
For every $\nnn$, $\upsigma(\boldsymbol{\varepsilon}_{\mathsf{n}})$
and $\upsigma(\boldsymbol{z}_{\lS},\boldsymbol{w}_{\lS})_{0\leq 
\lS\leq\nS}$ are independent.
\item
\label{p:2iii}
For every $\lS\in\{1,\ldots,\mathsf{p}+1+\mathsf{r}\}$,
$\PP[\varepsilon_{\lS,0}=1]>0$.
\end{enumerate}
Then $(x_{1,\nS})_{\nnn}$ converges weakly $\Pas$ to a 
$\ZS$-valued random variable. 
\end{proposition}
\begin{proof}
In Problem~\ref{prob:2}, set $\mathsf{m}=\mathsf{p}+1$, 
$\XS_1=\HS$,
$(\XS_{\iS})_{2\leq\iS\leq\mathsf{m}}=
(\GS_{\iS-1})_{2\leq\iS\leq\mathsf{m}}$, $\YYS=\KKS$,
for every $\jS\in\{1,\dots,\mathsf{r}\}$, 
$\mathsf{D}_{\jS}=\mathsf{N}_{\{\mathsf{0}\}}$, and, for every 
$\iS\in\{1,\dots,\mathsf{m}\}$, $\mathsf{M}_{\jS\iS}=
\mathsf{E}_{\jS\iS}$.
Thus, the subspace $\boldsymbol{\mathsf{V}}$ of \eqref{e:2014}
becomes
\begin{equation}
\label{e:20114}
\boldsymbol{\mathsf{V}}=
\Menge3{\boldsymbol{(\mathsf{x}},\boldsymbol{\mathsf{y}})
\in\XXS\oplus\YYS}
{(\forall\jS\in\{1,\ldots,\mathsf{r}\})\;\mathsf{y}_{\jS}=
\sum_{\iS=1}^{\mathsf{p}+1}\mathsf{E}_{\jS\iS}
{\mathsf{x}_{\iS}}},
\end{equation}
Further, denote by 
$\boldsymbol{\mathsf{x}}=(\mathsf{x}_{1},\dots,
\mathsf{x}_{\mathsf{p}+1})$ a generic element in $\HS\oplus\GGS$.
In this configuration, \eqref{e:p2} reduces to 
\begin{equation}
\label{e:p211}
\text{find}\;\;\boldsymbol{\mathsf{x}}\in\HS\oplus\GGS\;\;
\text{such that}\;\;\boldsymbol{\mathsf{0}}\in
\overset{\mathsf{p}+1}{\underset{\iS=1}{\bigtimes}}\;
\mathsf{C}_{\iS}\mathsf{x}_{\mathsf{i}}
+\boldsymbol{\mathsf{E}}^*
\brk1{\mathsf{N}_{\{\boldsymbol{\mathsf{0}}\}}
(\boldsymbol{\mathsf{E}}\boldsymbol{\mathsf{x}})}.
\end{equation} 
We note that Proposition~\ref{p:2} is the application of 
Theorem~\ref{t:1} to \eqref{e:p211} when
\begin{equation}
\label{e:v19831}
\mathsf{C}_{1}=\mathsf{A}\quad\text{and}\quad
(\forall\iS\in\{2,\ldots,\mathsf{p}+1\})\quad
\mathsf{C}_{\iS}=\mathsf{B}_{\iS-1}. 
\end{equation}
Let $\boldsymbol{\mathsf{Z}}$ be the set of solutions to
\eqref{e:p211} in the context of \eqref{e:v19831}.
Recalling that $\mathsf{Z}$ denotes the set of solutions to
Problem~\ref{prob:1}, it remains to show that
\begin{equation}
\boldsymbol{\mathsf{Z}}=\menge{
(\mathsf{x}_1,\LS_1\mathsf{x}_1,\dots,
\LS_{\mathsf{p}}\mathsf{x}_1)}{\mathsf{x}_1\in\mathsf{Z}}.
\end{equation}
Let $\boldsymbol{\mathsf{x}}\in\HS\oplus\GGS$. It follows at once
from \eqref{e:67} that
\begin{equation}
\iota_{\boldsymbol{\mathsf{W}}}(\boldsymbol{\mathsf{x}})=
\iota_{\{\boldsymbol{\mathsf{0}}\}}(\boldsymbol{\mathsf{E}}
\boldsymbol{\mathsf{x}}).
\end{equation}
Hence, we deduce from \cite[Corollary~16.53]{Livre1} that 
\begin{equation}
\label{e:NC}
\mathsf{N}_{\boldsymbol{\mathsf{W}}}\boldsymbol{\mathsf{x}}
=\boldsymbol{\mathsf{E}}^*
\brk1{\mathsf{N}_{\{\boldsymbol{\mathsf{0}}\}}
(\boldsymbol{\mathsf{E}}\boldsymbol{\mathsf{x}})}.
\end{equation}
Note that the set in \eqref{e:NC} is nonempty if and only if
$\boldsymbol{\mathsf{x}}\in\boldsymbol{\mathsf{W}}$.
Consequently,
\begin{align}
\boldsymbol{\mathsf{x}}\in
\boldsymbol{\mathsf{Z}}
\Leftrightarrow
\boldsymbol{\mathsf{x}}\;\text{solves \eqref{e:p211}}
\Leftrightarrow
\boldsymbol{\mathsf{x}}\;\text{solves \eqref{e:p21}}
\end{align}
and the claim follows from \eqref{e:999}.
\end{proof}

Let us provide some examples of implementations of 
Proposition~\ref{p:2}.

\begin{example}
\label{ex:11}
In Proposition~\ref{p:2}, set $\mathsf{r}=\mathsf{p}$, 
$\boldsymbol{\mathsf{K}}=\GGS$, and, for every 
$\kS\in\{1,\dots,\mathsf{p}\}$ and every 
$\iS\in\{1,\dots,\mathsf{p}+1\}$, 
\begin{equation}
\label{e:no}
\mathsf{E}_{\kS\iS}=
\begin{cases}
\phantom{-}\LS_{\kS}, & \text{if}\;\;\iS=1;\\
-\Id,&\text{if}\;\;\iS=\kS+1;\\
\phantom{-}\mathsf{0},& \text{otherwise.}
\end{cases}
\end{equation}
Let $\boldsymbol{\mathsf{x}}\in\HS\oplus\GGS$, 
let $\boldsymbol{\mathsf{y}}\in\GGS$, and set 
$\mathsf{q}=(2\Id+\sum_{\kS=1}^{\mathsf{p}}
\LS_{\kS}^*\circ\LS_{\kS})^{-1}
(2\mathsf{x}_1+
\sum_{\kS=1}^{\mathsf{p}}\LS_{\kS}^*(\mathsf{x}_{\kS+1}
+\mathsf{y}_{\kS}))$.
Then, for every $\iS\in\{1,\dots,\mathsf{p}+1\}$,
\begin{equation}
\mathsf{R}_{\iS}(\boldsymbol{\mathsf{x}},\boldsymbol{\mathsf{y}})=
\begin{cases}
\mathsf{q},
&\text{if}\;\;\iS=1;\\[2mm]
\dfrac{1}{2}\bigl(\LS_{\iS-1}\mathsf{q}
+\mathsf{x}_{\iS}-\mathsf{y}_{\iS-1}\bigr),
&\text{if}\;\;2\leq\iS\leq \mathsf{p}+1;\\[2mm]
\dfrac{1}{2}\bigl(\LS_{\iS-\mathsf{p}-1}
\mathsf{q}-\mathsf{x}_{\iS-\mathsf{p}}
+\mathsf{y}_{\iS-\mathsf{p}-1}\bigr),
&\text{if}\;\;\mathsf{p}+2\leq\iS\leq 2\mathsf{p}+1.
\end{cases}
\end{equation}
Let $(e_{\nS})_{\nnn}$ be $\HS$-valued random variables such that
$\sum_{\nnn}\sqrt{\EC{\|e_{\nS}\|^2}{\upsigma(\boldsymbol{z}_{\lS},
\boldsymbol{w}_{\lS})_{0\leq\lS\leq\nS}}}<\pinf$ and 
$e_{\nS}\weakly\mathsf{0}$ $\Pas$ and set 
\begin{equation}
\label{e:Q}
\mathsf{Q}=\brk3{2\Id+\sum_{\kS=1}^{\mathsf{p}}
\LS_{\kS}^*\circ\LS_{\kS}}^{-1}. 
\end{equation}
Then \eqref{e:algo3} becomes
\begin{equation}
\label{e:ex11}
\begin{array}{l}
\text{for}\;\nS=0,1,\ldots\\
\left\lfloor
\begin{array}{l}
q_{\nS}=
\mathsf{Q}\brk2{2{z}_{1,\nS}+\Sum_{\kS=1}^{\mathsf{p}}
\LS_{\kS}^*({z}_{\kS+1,\nS}+{w}_{\kS,\nS})}+e_{\nS}\\
{x}_{1,\nS+1}=
{x}_{1,\nS}+\varepsilon_{1,\nS}(q_{\nS}-x_{1,\nS})\\
z_{1,\nS+1}=
z_{1,\nS}+\varepsilon_{1,\nS}\uplambda_{\nS}
\bigl(\mathsf{J}_{\upgamma \mathsf{A}}
\left(2x_{1,\nS+1}-z_{1,\nS}\right)+c_{1,\nS}-x_{1,\nS+1}\bigr)\\
\text{for}\;\kS=1,\dots,\mathsf{p}\\
\left\lfloor
\begin{array}{l}
{x}_{\kS+1,\nS+1}=
{x}_{\kS+1,\nS}+\varepsilon_{\kS+1,\nS}
\Bigl(\dfrac{\LS_{\kS}q_{\nS}+z_{\kS+1,\nS}-w_{\kS,\nS}}{2}-
x_{\kS+1,\nS}\Bigr)\\[3mm]
z_{\kS+1,\nS+1}=
z_{\kS+1,\nS}+\varepsilon_{\kS+1,\nS}\uplambda_{\nS}
\bigl(\mathsf{J}_{\upgamma\mathsf{B}_{\kS}}
\left(2x_{\kS+1,\nS+1}-z_{\kS+1,\nS}\right)+c_{\kS+1,\nS}-
x_{\kS+1,\nS+1}\bigr)\\
\end{array}
\right.\\
\text{for}\;\kS=1,\dots,\mathsf{p}\\
\left\lfloor
\begin{array}{l}
{y}_{\kS,\nS+1}=
{y}_{\kS,\nS}+\varepsilon_{\mathsf{p}+1+\kS,\nS}
\Bigl(\dfrac{\LS_{\kS}q_{\nS}-z_{\kS+1,\nS}+w_{\kS,\nS}}{2}-
y_{\kS,\nS}\Bigr)\\[3mm]
{w}_{\kS,\nS+1}=
{w}_{\kS,\nS}-\varepsilon_{\mathsf{p}+1+\kS,\nS}
\uplambda_{\nS}y_{\kS,\nS+1}
\end{array}
\right.\\
\end{array}
\right.\\
\end{array}
\end{equation}
and Proposition~\ref{p:2} asserts that $(x_{1,\nS})_{\nnn}$
converges weakly $\Pas$ to a solution to Problem~\ref{prob:1}.
\end{example}

The next examples focus on the special case of 
Problem~\ref{prob:1} in which,
for every $\kS\in\{1,\dots,\mathsf{p}\}$, $\GS_{\kS}=\HS$ and 
$\LS_{\kS}=\Id$, that is,
\begin{equation}
\label{prob:noL}
\text{find}\;\;{\mathsf{x}\in\HS}\;\;\text{such that}\;\;
{\mathsf{0}\in\mathsf{A}\mathsf{x}+\sum_{\kS=1}^{\mathsf{p}}
\mathsf{B}_{\kS}\mathsf{x}}.
\end{equation}

\begin{example}
\label{ex:12}
Consider the setting of Example~\ref{ex:11} where, for every
$\kS\in\{1,\dots,\mathsf{p}\}$, $\GS_{\kS}=\HS$ and 
$\LS_{\kS}=\Id$. Then, in view of \eqref{e:no} the operator 
$\boldsymbol{\mathsf{E}}$ is
defined by setting, by for every 
$\kS\in\{1,\dots,\mathsf{p}\}$ and every 
$\iS\in\{1,\dots,\mathsf{p}+1\}$, 
\begin{equation}
\mathsf{E}_{\kS\iS}=
\begin{cases}
\phantom{-}\Id, & \text{if}\;\;\iS=1;\\
-\Id,&\text{if}\;\;\iS=\kS+1;\\
\phantom{-}\mathsf{0},& \text{otherwise.}
\end{cases}
\end{equation}
Further, the operator $\mathsf{Q}$ of \eqref{e:Q} is just 
$(\mathsf{p}+2)^{-1}\Id$. Thus, \eqref{e:algo3} becomes
\begin{equation}
\label{e:ex12}
\begin{array}{l}
\text{for}\;\nS=0,1,\ldots\\
\left\lfloor
\begin{array}{l}
q_{\nS}=\dfrac{1}{\mathsf{p}+2}
\brk2{2{z}_{1,\nS}+\Sum_{\kS=1}^{\mathsf{p}}
({z}_{\kS+1,\nS}+{w}_{\kS,\nS})}\\
{x}_{1,\nS+1}=
{x}_{1,\nS}+\varepsilon_{1,\nS}(q_{\nS}-x_{1,\nS})\\
z_{1,\nS+1}=
z_{1,\nS}+\varepsilon_{1,\nS}\uplambda_{\nS}
\bigl(\mathsf{J}_{\upgamma \mathsf{A}}
\left(2x_{1,\nS+1}-z_{1,\nS}\right)+c_{1,\nS}-x_{1,\nS+1}\bigr)\\
\text{for}\;\kS=1,\dots,\mathsf{p}\\
\left\lfloor
\begin{array}{l}
{x}_{\kS+1,\nS+1}=
{x}_{\kS+1,\nS}+\varepsilon_{\kS+1,\nS}
\Bigl(\dfrac{q_{\nS}+z_{\kS+1,\nS}-w_{\kS,\nS}}{2}-
x_{\kS+1,\nS}\Bigr)\\[3mm]
z_{\kS+1,\nS+1}=
z_{\kS+1,\nS}+\varepsilon_{\kS+1,\nS}\uplambda_{\nS}
\bigl(\mathsf{J}_{\upgamma\mathsf{B}_{\kS}}
\left(2x_{\kS+1,\nS+1}-z_{\kS+1,\nS}\right)+c_{\kS+1,\nS}-
x_{\kS+1,\nS+1}\bigr)\\
\end{array}
\right.\\
\text{for}\;\kS=1,\dots,\mathsf{p}\\
\left\lfloor
\begin{array}{l}
{y}_{\kS,\nS+1}=
{y}_{\kS,\nS}+\varepsilon_{\mathsf{p}+1+\kS,\nS}
\Bigl(\dfrac{q_{\nS}-z_{\kS+1,\nS}+w_{\kS,\nS}}{2}-
y_{\kS,\nS}\Bigr)\\[3mm]
{w}_{\kS,\nS+1}=
{w}_{\kS,\nS}-\varepsilon_{\mathsf{p}+1+\kS,\nS}
\uplambda_{\nS}y_{\kS,\nS+1}
\end{array}
\right.\\
\end{array}
\right.\\
\end{array}
\end{equation}
and Proposition~\ref{p:2} asserts that $(x_{1,\nS})_{\nnn}$
converges weakly $\Pas$ to a solution to \eqref{prob:noL}. 
\end{example}

\begin{example}
\label{ex:13}
In Proposition~\ref{p:2}, set $\mathsf{r}=\mathsf{p}+1$, 
$\boldsymbol{\mathsf{K}}=\HS^{\mathsf{p}+1}$, and, for every 
$\kS\in\{1,\dots,\mathsf{p}+1\}$ and every 
$\iS\in\{1,\dots,\mathsf{p}+1\}$, 
\begin{equation}
\mathsf{E}_{\kS\iS}=
\begin{cases}
\phantom{-}\Frac{\mathsf{p}}{\mathsf{p}+1}\Id,
&\text{if}\;\;\kS=\iS;\\[3mm]
-\Frac{1}{\mathsf{p}+1}\Id,&\text{if}\;\;\kS\neq\iS.
\end{cases}
\end{equation}
Then $\ker\boldsymbol{\mathsf{E}}$ is the subspace of all the
vectors ${\boldsymbol{\mathsf{x}}\in\HS^{\mathsf{p}+1}}$ such that,
for every $\iS\in\{1,\dots,\mathsf{p}+1\}$, 
$\mathsf{x}_{\iS}=(\mathsf{p}+1)^{-1}\sum_{\jS=1}^{\mathsf{p}+1}
\mathsf{x}_{\jS}$. Hence, for every 
$\iS\in\{1,\dots,2\mathsf{p}+2\}$,
every $\boldsymbol{\mathsf{x}}\in\HS^{\mathsf{p}+1}$, and 
every $\boldsymbol{\mathsf{y}}\in\HS^{\mathsf{p}+1}$, 
\begin{equation}
\mathsf{R}_{\iS}(\boldsymbol{\mathsf{x}},\boldsymbol{\mathsf{y}})=
\begin{cases}
\dfrac{1}{2}(\mathsf{x}_{\iS}+\mathsf{y}_{\iS})
+\dfrac{1}{2(\mathsf{p}+1)}
\Sum_{\jS=1}^{\mathsf{p}+1}(\mathsf{x}_{\jS}-\mathsf{y}_{\jS}),
&\text{if}\;\;\iS\leq\mathsf{p}+1;\\[1mm]
\dfrac{1}{2}(\mathsf{x}_{\iS}+\mathsf{y}_{\iS})
-\dfrac{1}{2(\mathsf{p}+1)}
\Sum_{\jS=1}^{\mathsf{p}+1}(\mathsf{x}_{\jS}+\mathsf{y}_{\jS}),
&\text{if}\;\;\mathsf{p}+2\leq\iS\leq2\mathsf{p}+2.
\end{cases}
\end{equation}
Then \eqref{e:algo3} becomes
\begin{equation}
\begin{array}{l}
\text{for}\;\nS=0,1,\ldots\\
\left\lfloor
\begin{array}{l}
{x}_{1,\nS+1}=
{x}_{1,\nS}+\varepsilon_{1,\nS}\Bigl(\dfrac{z_{1,\nS}+w_{1,\nS}}{2}
+\dfrac{1}{2(\mathsf{p}+1)}\Sum_{\lS=1}^{\mathsf{p}+1}
(z_{\lS,\nS}-w_{\lS,\nS})-x_{1,\nS}\Bigr)\\[3mm]
z_{1,\nS+1}=
z_{1,\nS}+\varepsilon_{1,\nS}\uplambda_{\nS}
\bigl(\mathsf{J}_{\upgamma \mathsf{A}}
\left(2x_{1,\nS+1}-z_{1,\nS}\right)+c_{1,\nS}-x_{1,\nS+1}\bigr)\\
\text{for}\;\kS=1,\dots,\mathsf{p}\\
\left\lfloor
\begin{array}{l}
{x}_{\kS+1,\nS+1}=
{x}_{\kS+1,\nS}+\varepsilon_{\kS+1,\nS}
\Bigl(\dfrac{z_{\kS+1,\nS}+w_{\kS+1,\nS}}{2}
+\dfrac{1}{2(\mathsf{p}+1)}\Sum_{\lS=1}^{\mathsf{p}+1}
(z_{\lS,\nS}-w_{\lS,\nS})-
x_{\kS+1,\nS}\Bigr)\\[3mm]
z_{\kS+1,\nS+1}=
z_{\kS+1,\nS}+\varepsilon_{\kS+1,\nS}\uplambda_{\nS}
\bigl(\mathsf{J}_{\upgamma\mathsf{B}_{\kS}}
\left(2x_{\kS+1,\nS+1}-z_{\kS+1,\nS}\right)+c_{\kS+1,\nS}-
x_{\kS+1,\nS+1}\bigr)\\
\end{array}
\right.\\
\text{for}\;\jS=1,\dots,\mathsf{p}+1\\
\left\lfloor
\begin{array}{l}
{y}_{\jS,\nS+1}=
{y}_{\jS,\nS}+\varepsilon_{\mathsf{p}+1+\jS,\nS}
\Bigl(\dfrac{z_{\jS,\nS}+w_{\jS,\nS}}{2}
-\dfrac{1}{2(\mathsf{p}+1)}\Sum_{\lS=1}^{\mathsf{p}+1}
(z_{\lS,\nS}+w_{\lS,\nS})-
y_{\jS,\nS}\Bigr)\\[3mm]
{w}_{\jS,\nS+1}=
{w}_{\jS,\nS}-\varepsilon_{\mathsf{p}+1+\jS,\nS}
\uplambda_{\nS}y_{\jS,\nS+1}
\end{array}
\right.\\
\end{array}
\right.\\
\end{array}
\end{equation}
and Proposition~\ref{p:2} asserts that $(x_{1,\nS})_{\nnn}$
converges weakly $\Pas$ to a solution to \eqref{prob:noL}. 
\end{example}

\begin{remark}
\label{r:5}
In Example~\ref{ex:12}, the operator $\boldsymbol{\mathsf{E}}$
applied to $\boldsymbol{\mathsf{x}}\in\HS^{\mathsf{p}+1}$ couples
each agent in $(\mathsf{x}_{2},\ldots,\mathsf{x}_{\mathsf{p}+1})$
with $\mathsf{x}_1$. In Example~\ref{ex:13} the operator
$\boldsymbol{\mathsf{E}}$ applied to
$\boldsymbol{\mathsf{x}}\in\HS^{\mathsf{p}+1}$ couples
each agent in $(\mathsf{x}_{1},\ldots,\mathsf{x}_{\mathsf{p}+1})$
with the average of all the agents. Various alternative coupling
operators $\boldsymbol{\mathsf{E}}$ can be considered to enforce
the condition $\mathsf{x}_1=\cdots=\mathsf{x}_{\mathsf{p}+1}$.
\end{remark}

\subsection{Computation of inverse operators}
\label{sec:35}

The existing algorithms presented in Section~\ref{sec:22} require
the computation of norms of arbitrary linear operators whereas the
proposed algorithms of Section~\ref{sec:f1}--\ref{sec:f3} require
the inversion of strongly positive Hermitian operators of the type
$\ID+\LLS^*\circ\LLS$. Note that, because of the strongly positive
hermitian structure of $\ID+\LLS^*\circ\LLS$, the computation of
the inverse is typically much cheaper than the computation of the
norm of $\LLS$ in \eqref{e:CHcond} or those of
$(\mathsf{U}_{\kS}^{1/2}\LS_{\kS}
\mathsf{W}^{1/2})_{1\leq\kS\leq\mathsf{p}}$ in \eqref{e:PRcond}. In
a finite dimension setting, in full generality, if
$\ID+\LLS^*\circ\LLS$ has size $\mathsf{N}$, its inversion via the
Cholesky decomposition method requires about $\mathsf{N}^3/6$
multiplications. However, this complexity can be reduced in several
standard scenarios. Here are two examples in
$\HS=\RR^{\mathsf{N}}$ that will be used in Section~\ref{sec:5}.

\begin{example}\
\label{ex:inv}
\begin{enumerate}
\item
\label{ex:invi}
If, for every $\kS\in\{1,\ldots,\mathsf{p}\}$, $\LS_{\kS}=\Id$.
Then
\begin{equation}
\label{e:inv1}
\begin{cases}
\brk2{\Id+\sum_{\kS=1}^{\mathsf{p}}\LS_{\kS}^*\circ\LS_{\kS}}^{-1}
=\dfrac{1}{1+\mathsf{p}}\Id\\[3mm]
\brk2{2\Id+\sum_{\kS=1}^{\mathsf{p}}\LS_{\kS}^*\circ\LS_{\kS}}^{-1}
=\dfrac{1}{2+\mathsf{p}}\Id.
\end{cases}
\end{equation}
The cost of the inversion is $O(1)$.
\item
\label{ex:invii}
Suppose that, for every $\kS\in\{1,\ldots,\mathsf{p}\}$,
$\LS_{\kS}$ is a block-Toeplitz. Then, following a standard
argument \cite{Andr77}, each $\LS_{\kS}$ can be approximated by a
block-circulant matrix with convolution kernel $\ell_{\kS}$ and
\begin{equation}
\label{e:inv2}
\begin{cases}
\brk3{\Id+\displaystyle\sum_{\kS=1}^{\mathsf{p}}
\LS_{\kS}^*\circ\LS_{\kS}}^{-1}
&\colon
\mathsf{x}\mapsto\mathfrak{F}^{-1}
\brk3{{\mathfrak{F}(\mathsf{x})}\div
\brk3{{1+\displaystyle\sum_{\kS=1}^{\mathsf{p}}
\abs{\mathfrak{F}(\ell_{\kS})}^2}}}\\[5mm]
\brk3{2\Id+\displaystyle\sum_{\kS=1}^{\mathsf{p}}
\LS_{\kS}^*\circ\LS_{\kS}}^{-1}
&\colon
\mathsf{x}\mapsto\mathfrak{F}^{-1}
\brk3{{\mathfrak{F}(\mathsf{x})}\div
\brk3{2+\displaystyle\sum_{\kS=1}^{\mathsf{p}}
\abs{\mathfrak{F}(\ell_{\kS})}^2}}.
\end{cases}
\end{equation}
where $\mathfrak{F}$ denotes the discrete Fourier transform and 
$\div$ denotes pointwise division. The cost of the inversion using
the fast Fourier transform is $O(\mathsf{N}\log(\mathsf{N}))$
\cite{Andr77}.
\item
\label{ex:inviii}
The worst case is if the operators
$(\LS_{\kS})_{1\leq\kS\leq\mathsf{p}}$ does not present a special
structure. Even so, the composed operators 
$\Id+\sum_{\kS=1}^{\mathsf{p}}\LS_{\kS}^*\circ\LS_{\kS}$ 
and
$2\Id+\sum_{\kS=1}^{\mathsf{p}}\LS_{\kS}^*\circ\LS_{\kS}$ are
symmetric and positive-definite. Hence they admit a Cholesky
decomposition. The cost of computing the Cholesky 
decomposition is $O(\mathsf{N}^3)$ (one time) and the cost of
solving the linear system using the Cholesky decomposition is 
$O(\mathsf{N}^2)$. It will be shown in the numerical experiments
that in the case when no special structure is present, Framework~2
is preferred since the application of the inverse operator does not
occur at every iteration. 
\end{enumerate}
\end{example}

\section{Minimization problems}
\label{sec:4}
We dedicate this section to the minimization setting of
Problem~\ref{prob:10}. Let us first formalize the connection
between Problem~\ref{prob:10} and Problem~\ref{prob:1}.

\begin{proposition}
\label{p:4}
In Problem~\ref{prob:1}, set $\mathsf{A}=\partial\mathsf{f}$ and 
$(\forall\kS\in\{1,\ldots,\mathsf{p}\})$
$\mathsf{B}_{\kS}=\partial\mathsf{g}_{\kS}$. Then every solution to
to \eqref{e:p1} solves Problem~\ref{prob:10}.
\end{proposition}
\begin{proof}
Set $\LLS\colon\HS\to\GGS\colon\mathsf{x}\mapsto(\LS_1\mathsf{x},
\ldots,\LS_{\mathsf{p}}\mathsf{x})$ and 
$\boldsymbol{\mathsf{g}}\colon\GGS\to\RX\colon
\boldsymbol{\mathsf{y}}\mapsto
\sum_{\kS=1}^{\mathsf{p}}\mathsf{g}_{\kS}(\mathsf{y}_{\kS})$.
Then $\LLS^*\colon\GGS\to\HS\colon\boldsymbol{\mathsf{y}}
\mapsto\sum_{\kS=1}^{\mathsf{p}}\LS_{\kS}^*\mathsf{y}_{\kS}$.
Hence, it follows from \cite[Proposition~16.9]{Livre1} that
\begin{equation}
\mathsf{x}\in\zer\brk3{\partial\mathsf{f}+
\sum_{\kS=1}^\mathsf{p}\LS^*_{\kS}\circ
(\partial\mathsf{g}_{\kS})\circ\LS_{\kS}}=
\zer\brk1{\partial\mathsf{f}+\LLS^*\circ
(\partial\boldsymbol{\mathsf{g}})\circ\LLS}.
\end{equation}
However, \cite[Proposition~27.5(i)]{Livre1} asserts that
\begin{equation}
\zer\brk3{\partial\mathsf{f}+\sum_{\kS=1}^\mathsf{p}\LLS^*\circ
(\partial\boldsymbol{\mathsf{g}})\circ\LLS}\subset
\Argmin\brk1{\mathsf{f}+\boldsymbol{\mathsf{g}}\circ\LLS}
=\Argmin\brk3{\mathsf{f}+\sum_{\kS=1}^\mathsf{p}
\mathsf{g}_{\kS}\circ\LS_{\kS}},
\end{equation}
which confirms the claim.
\end{proof}

Problem~\ref{prob:10} relies on the assumption that 
$\zer(\partial\mathsf{f}+\sum_{\kS=1}^\mathsf{p}\LS^*_{\kS}
\circ(\partial\mathsf{g}_{\kS})\circ\LS_{\kS})\neq\emp$. Let us
provide sufficient conditions that guarantee it.

\begin{proposition}
\label{p:6}
Let $\HS$ be a separable real Hilbert space and 
$\mathsf{f}\in\upGamma_0(\HS)$. For every
$\kS\in\{1,\ldots,\mathsf{p}\}$, let $\GS_{\kS}$ be a separable
real Hilbert space, let $\mathsf{g}_{\kS}\in\upGamma_0(\GS_{\kS})$,
and let $0\neq\LS_{\kS}\colon\HS\to\GS_{\kS}$ be linear and
bounded. Set
\begin{equation}
\boldsymbol{\mathsf{S}}=\menge{(\LS_{\kS}\mathsf{x}-
\mathsf{y}_{\kS})_{1\leq\kS\leq\mathsf{p}}}
{\mathsf{x}\in\dom\mathsf{f}\;\text{and}\;
(\forall\kS\in\{1,\ldots,\mathsf{p}\})
\;\:\mathsf{y}_{\kS}\in\dom\mathsf{g}_{\kS}}.
\end{equation}
Then $\zer(\partial\mathsf{f}+\sum_{\kS=1}^\mathsf{p}\LS^*_{\kS}
\circ(\partial\mathsf{g}_{\kS})\circ\LS_{\kS})\neq\emp$ if the
following hold:
\begin{enumerate}
\item
\label{p:6i}
$\mathsf{f}(\mathsf{x})+\sum_{\kS=1}^\mathsf{p}
\mathsf{g}_{\kS}(\LS_{\kS}\mathsf{x})\to\pinf$ as 
$\|\mathsf{x}\|\to\pinf$.
\item
\label{p:6ii}
Any of the following is satisfied:
\begin{enumerate}
\item
\label{p:6iia}
The cone generated by $\boldsymbol{\mathsf{S}}$ is a closed vector
subspace of \,$\GGS$.
\item
\label{p:6iib}
For every $\kS\in\{1,\ldots,\mathsf{p}\}$, $\mathsf{g}_{\kS}$ is
real-valued.
\item
\label{p:6iic}
$\HS$ and $(\GS_{\kS})_{1\leq\kS\leq\mathsf{p}}$ are
finite-dimensional, and there exists
$\mathsf{x}\in\reli\dom\mathsf{f}$ such that 
\begin{equation}
(\forall\kS\in\{1,\ldots,\mathsf{p}\})\quad\LS_{\kS}\mathsf{x}\in
\reli\dom\mathsf{g}_{\kS},
\end{equation}
where $\reli$ stands for the relative interior.
\end{enumerate}
\end{enumerate}
\end{proposition}
\begin{proof}
Set $\LLS\colon\HS\to\GGS\colon\mathsf{x}\mapsto(\LS_1\mathsf{x},
\ldots,\LS_{\mathsf{p}}\mathsf{x})$ and 
$\boldsymbol{\mathsf{g}}\colon\GGS\to\RX
\colon\boldsymbol{\mathsf{y}}\mapsto
\sum_{\kS=1}^{\mathsf{p}}\mathsf{g}_{\kS}(\mathsf{y}_{\kS})$. Then 
$\LLS$ is linear and bounded, 
$\boldsymbol{\mathsf{g}}\in\upGamma_0(\GGS)$, 
$\boldsymbol{\mathsf{S}}=
\menge{\LLS\mathsf{x}-\boldsymbol{\mathsf{y}}}
{\mathsf{x}\in\dom\mathsf{f}\;\text{and}\;
\boldsymbol{\mathsf{y}}\in\dom\boldsymbol{\mathsf{g}}}$, and 
$\mathsf{f}+\boldsymbol{\mathsf{g}}\circ\LLS=
\mathsf{f}+\sum_{\kS=1}^\mathsf{p}\mathsf{g}_{\kS}\circ\LS_{\kS}$.
On the other hand, it follows from \ref{p:6ii} that
$\boldsymbol{\mathsf{0}}\in\boldsymbol{\mathsf{S}}$, which implies
that $\dom(\mathsf{f}+\boldsymbol{\mathsf{g}}\circ\LLS)\neq\emp$. 
Thus, because $\mathsf{f}+\boldsymbol{\mathsf{g}}\circ\LLS$ is also
lower semicontinuous and convex, we have
$\mathsf{f}+\boldsymbol{\mathsf{g}}\circ\LLS\in\upGamma_0(\HS)$. 
Hence, since \ref{p:6i} states that 
$\mathsf{f}(\mathsf{x})+\boldsymbol{\mathsf{g}}(\LLS\mathsf{x})
\to\pinf$ as $\|\mathsf{x}\|\to\pinf$, it follows from 
\cite[Proposition~11.15(i)]{Livre1} that 
\begin{equation}
\label{e:t6}
\Argmin(\mathsf{f}+\boldsymbol{\mathsf{g}}\circ\LLS)\neq\emp.
\end{equation}
However, \ref{p:6ii} and \cite[Proposition~27.5(iii)]{Livre1} 
guarantee that
\begin{equation}
\Argmin(\mathsf{f}+\boldsymbol{\mathsf{g}}\circ\LLS)=
\zer\brk1{\partial\mathsf{f}+\LLS^*
\circ(\partial\boldsymbol{\mathsf{g}})\circ\LLS}, 
\end{equation}
which completes the proof.
\end{proof}

In view of Proposition~\ref{p:4} and \eqref{e:dprox}, we obtain the
following solution methods for Problem~\ref{prob:10}.

\begin{corollary}
\label{c:1}
Consider the setting of Problem~\ref{prob:10} and set
$\mathsf{F}=\Argmin(\mathsf{f}+\sum_{\kS=1}^\mathsf{p}
\mathsf{g}_{\kS}\circ\LS_{\kS})$. In \eqref{e:algo1},
\eqref{e:algo2}, and \eqref{e:algo3}, replace the resolvent
operators $(\mathsf{J}_{\upgamma\mathsf{A}},
\mathsf{J}_{\upgamma\mathsf{B}_1},\ldots,
\mathsf{J}_{\upgamma\mathsf{B}_{\mathsf{p}}})$ by the proximity
operators $(\prox_{\upgamma\mathsf{f}},
\prox_{\upgamma\mathsf{g}_1},\ldots,
\prox_{\upgamma\mathsf{g}_{\mathsf{p}}})$.
Then Propositions~\ref{p:3}, \ref{p:21}, and
\ref{p:2} provide sequences $(x_{1,\nS})_{\nnn}$ which
converges weakly $\Pas$ to an $\mathsf{F}$-valued random variable.
\end{corollary}

\section{Numerical experiments}
\label{sec:5}

\subsection{Preamble}
\label{sec:50}

We present four experiments to illustrate the numerical behavior of
the three algorithmic frameworks presented in Section~\ref{sec:3}.
These algorithms are initialized by setting $\boldsymbol{x}_0$,
$\boldsymbol{z}_0$, $\boldsymbol{y}_0$, and $\boldsymbol{w}_0$ to
$\boldsymbol{\mathsf{0}}$, and they use the proximal parameter
$\upgamma=1.0$ and the relaxation strategy
$(\forall\nnn)$ $\uplambda_\nS=1.9$. 
The random variable $\boldsymbol{\varepsilon}_0$ activates
operator indices in $\{1,\dots,\mathsf{p}+1\}$ (Framework~1), 
$\{1,\dots,\mathsf{p}+2\}$ (Framework~2), and
$\{1,\dots,2\mathsf{p}+1\}$ (Framework~3 using 
Example~\ref{ex:11}), with a uniform distribution.

We also provide comparisons
with the existing methods of Section~\ref{sec:22} when applicable,
because they do provide almost sure iterate convergence to a
solution, although they do not satisfy the requirements
{\bfseries R2}--{\bfseries R3}:
\begin{itemize}
\item
Algorithm~\eqref{e:72} is initialized with ${x}_{1,0}=\mathsf{0}$
and $\boldsymbol{y}_0=\boldsymbol{\mathsf{0}}$. Further, 
for every $\kS\in\{1,\dots,\mathsf{p}\}$, 
$\uppi_{\kS}=1/\mathsf{p}$ and, to enforce \eqref{e:CHcond}, we set
$\tau_0=0.9/\sqrt{\mathsf{p}}$ and 
$\sigma_0=1/\brk1{\sqrt{\mathsf{p}}
\max_{1\leq\kS\leq\mathsf{p}}\|\LS_{\kS}\|^2}$. In addition we set
$\chi_0=0.5$, $\upeta=0.5$, and $\updelta=1.5$. We recall that
Algorithm~\eqref{e:72} can activate only one operator at each
iteration and does not satisfy {\bfseries R2}--{\bfseries R4}.
\item
Algorithm~\eqref{e:71} is initialized with 
${x}_{1,0}=\mathsf{0}$ and 
$\boldsymbol{v}_0=\boldsymbol{\mathsf{0}}$. Further, 
$\mathsf{W}=0.9\uptau\Id$ and, for every
$\kS\in\{1,\dots,\mathsf{p}\}$, 
$\mathsf{U}_{\kS}=(\uptau/\|\LS_{\kS}\|^2)\Id$,
where $\uplambda_{\nS}\equiv1$ and, to enforce \eqref{e:PRcond},
$\uptau=1/\sqrt{2\mathsf{p}}$. We recall that
Algorithm~\eqref{e:71} does not satisfy 
{\bfseries R2}--{\bfseries R3}.
\end{itemize}
These parameters were found to enhance the performance of these
two algorithms. The first three experiments
(Sections~\ref{sec:51}--\ref{sec:53}) correspond to minimization
problems fitting the format of Problem~\ref{prob:10}. The last
experiment (Section~\ref{sec:54}) is a non-minimization problem
that fits the format of Problem~\ref{prob:1}, and 
Algorithm~\eqref{e:72} is therefore not applicable. 

\subsection{Signal restoration}
\label{sec:51}

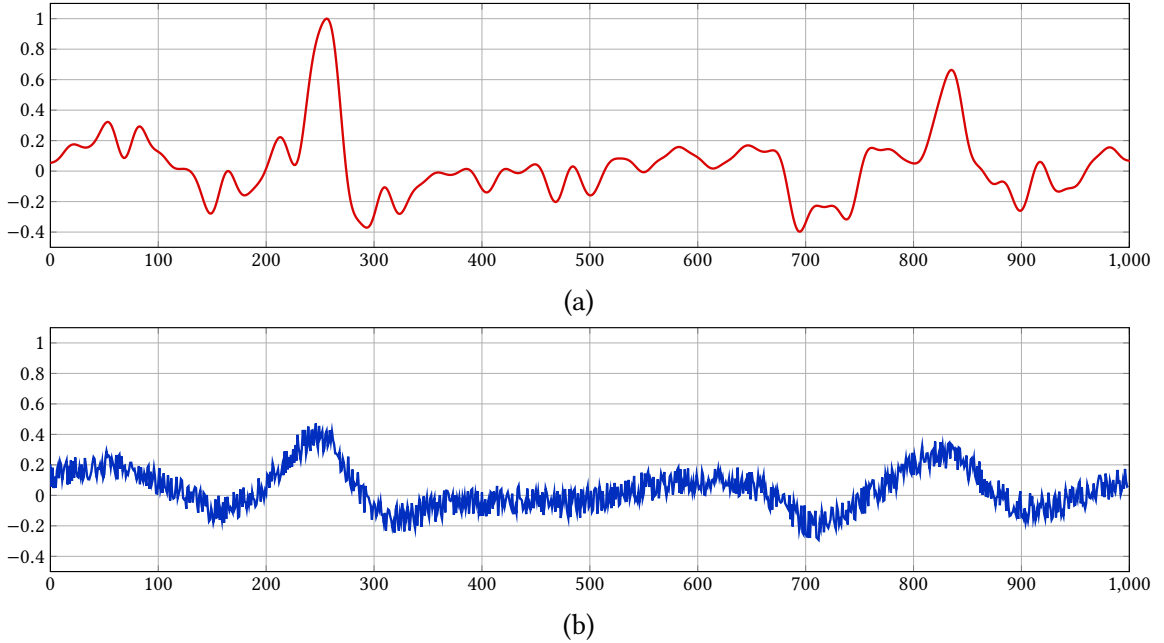
\begin{figure}[b!]
\centering
\begin{tikzpicture}[scale=0.545]
\definecolor{darkgray176}{RGB}{176,176,176}
\begin{axis}[height=7.5cm,width=1.818\columnwidth, legend 
columns=1 
cell align={left}, xmin=0, xmax=1000, ymin=-0.5, ymax=1.1,
xtick distance=100,
ytick distance=0.2,
x grid style={darkgray176},xmajorgrids,
y grid style={darkgray176},ymajorgrids,
tick label style={font=\Large}, 
legend cell align={left},
legend style={at={(0.02,0.15)},anchor=west},
axis line style=thick,]
\addplot [ultra thick, dred]
table {figures/ex1/original.txt};
\end{axis}
\end{tikzpicture}\\
(a)\\[0.3em]
\begin{tikzpicture}[scale=0.545]
\definecolor{darkgray176}{RGB}{176,176,176}
\begin{axis}[height=7.5cm,width=1.818\columnwidth, legend 
columns=1 
cell align={left}, xmin=0, xmax=1000, ymin=-0.5, ymax=1.1,
xtick distance=100,
ytick distance=0.2,
x grid style={darkgray176},xmajorgrids,
y grid style={darkgray176},ymajorgrids,
tick label style={font=\Large}, 
legend cell align={left},
legend style={at={(0.02,0.15)},anchor=west},
axis line style=thick,]
\addplot [ultra thick, dblue]
table {figures/ex1/err1.txt};
\end{axis}
\end{tikzpicture}\\
(b)\\[0.3em]
\caption{Experiment of Section~\ref{sec:51}.
(a): Original signal $\overline{\mathsf{x}}$.
(b): Noisy observation $\mathsf{r}_1$.}
\label{fig:sign}
\end{figure}

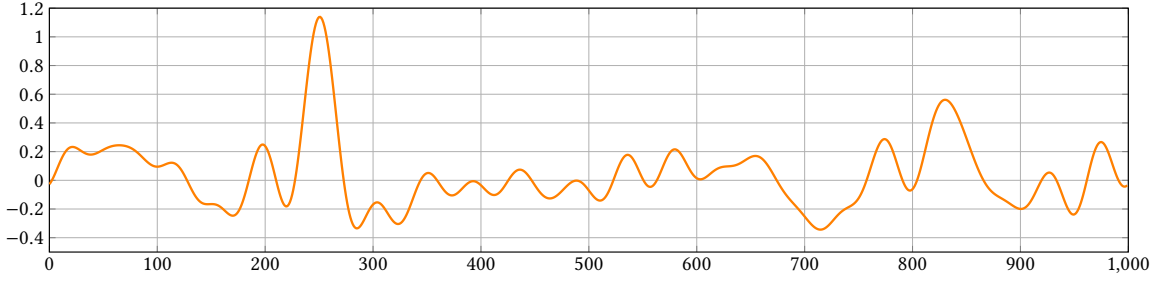
\begin{figure}[t!]
\begin{tikzpicture}[scale=0.545]
\definecolor{darkgray176}{RGB}{176,176,176}
\begin{axis}[height=7.5cm,width=1.818\columnwidth, 
legend columns=1 
cell align={left}, xmin=0, xmax=1000, ymin=-0.5, ymax=1.2,
xtick distance=100,
ytick distance=0.2,
x grid style={darkgray176},xmajorgrids,
y grid style={darkgray176},ymajorgrids,
tick label style={font=\Large}, 
legend cell align={left},
legend style={at={(0.02,0.15)},anchor=west},
axis line style=thick,]
\addplot [ultra thick, orange]
table {figures/ex1/sol2b.txt};
\end{axis}
\end{tikzpicture}
\caption{Experiment of Section~\ref{sec:51}. Solution
produced by Framework~2.}
\label{fig:sol2}
\end{figure}

\begin{figure}[b!]
\centering
\begin{tabular}{c@{}c@{}}
\begin{tikzpicture}[scale=0.545]
\definecolor{darkgray176}{RGB}{176,176,176}
\begin{axis}[height=8cm,width=13cm, legend columns=1 
cell align={left}, xmin=0, xmax=1000, ymin=-100, ymax=0.00,
xtick distance=250,
ytick distance=25,
x grid style={darkgray176},xmajorgrids,
y grid style={darkgray176},ymajorgrids,
tick label style={font=\Large}, 
legend cell align={left},
legend style={at={(0.02,0.15)},anchor=west},
axis line style=thick,]
\addplot [ultra thick, ngreen]
table {figures/ex1/F1b1.txt};
\addplot [ultra thick, orng]
table {figures/ex1/F2b1.txt};
\addplot [ultra thick, dblue]
table {figures/ex1/F3b1.txt};
\addplot [ultra thick, dashed, dviolet]
table {figures/ex1/CHb1.txt};
\addplot [ultra thick, dashed, dred]
table {figures/ex1/RNb1.txt};
\end{axis}
\end{tikzpicture}&
\begin{tikzpicture}[scale=0.545]
\definecolor{darkgray176}{RGB}{176,176,176}
\begin{axis}[height=8cm,width=13cm, legend columns=1 
cell align={left}, xmin=0, xmax=250, ymin=-100, ymax=0.00,
xtick distance=50,
ytick distance=25,
x grid style={darkgray176},xmajorgrids,
y grid style={darkgray176},ymajorgrids,
tick label style={font=\Large}, 
legend cell align={left},
legend style={at={(0.02,0.15)},anchor=west},
axis line style=thick,]
\addplot [ultra thick, ngreen]
table {figures/ex1/F1b8.txt};
\addplot [ultra thick, orng]
table {figures/ex1/F2b8.txt};
\addplot [ultra thick, dblue]
table {figures/ex1/F3b8.txt};
%\addplot [ultra thick, dashed, dviolet]
%table {figures/ex1/CHb1.txt};
\addplot [ultra thick, dashed, dred]
table {figures/ex1/RNb8.txt};
\end{axis}
\end{tikzpicture}\\
\small{(a)}&\small{(b)}
\end{tabular}
\caption{Experiment of Section~\ref{sec:51}.
Normalized error $20\log_{10}(\|x_{1,\nS}-x_\infty\|/
\|x_{1,0}-x_\infty\|)$ (dB) versus execution 
time (s).
(a): Block size $1$ with $1$ core. 
(b): Block size $8$ with $8$ cores. 
{\color{ngreen} Green}: Framework~1.
{\color{orng} Orange}: Framework~2. 
{\color{dblue} Blue}: Framework~3
with Example~\ref{ex:11}.
{\color{dviolet} Dashed violet}: Algorithm~\eqref{e:72}.
{\color{dred} Dashed red}: Algorithm~\eqref{e:71}.
}
\label{fig:ex3.2}
\end{figure}
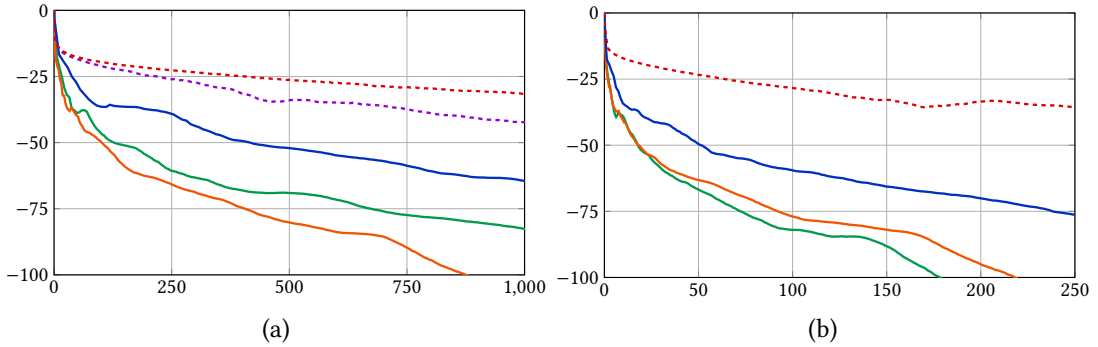
The goal is to recover the original signal 
$\overline{\mathsf{x}}\in\HS=\RR^\mathsf{N}$ ($\mathsf{N}=1000$) 
shown in Figure~\ref{fig:sign}(a) from $\mathsf{M}=10$ noisy
observations 
$(\mathsf{r}_{\lS})_{1\leq\lS\leq\mathsf{M}}$ given by
\begin{equation}
\label{e:exp1}
(\forall\lS\in\{1,\dots,\mathsf{M}\})\quad
\mathsf{r}_{\lS}=\LS_{\lS}\overline{\mathsf{x}}+\mathsf{w}_{\lS}
\end{equation}
where, for every
$\lS\in\{1,\dots,\mathsf{M}\}$,
$\LS_{\lS}\colon\RR^\mathsf{N}\to\RR^\mathsf{N}$ is a known linear 
operator, $\upeta_{\lS}\in\RPP$, and 
$\mathsf{w}_{\lS}\in
\left[-\upeta_{\lS},\upeta_{\lS}\right]^\mathsf{N}$ is the
realization of a bounded random noise vector. The parameters
$(\upeta_{\lS})_{1\leq\lS\leq\mathsf{M}}\in\RPP^\mathsf{M}$ are not
known exactly and underestimated by 
$(\upxi_{\lS})_{1\leq\lS\leq\mathsf{M}}\in\RPP^\mathsf{M}$. 
For every $\lS\in\{1,\dots,\mathsf{M}\}$, $\LS_{\lS}$ 
is a Gaussian convolution filter with zero mean and standard
deviation taken uniformly in $[20,40]$, $\upeta_{\lS}=0.1$,
$\mathsf{w}_{\lS}$ is taken uniformly
in $\left[-\upeta_{\lS},\upeta_{\lS}\right]^\mathsf{N}$, and
$\upxi_{\lS}=0.07$.
Set, for every $\lS\in\{1,\dots,\mathsf{M}\}$ and every
$\jS\in\{1,\dots,\mathsf{N}\}$, $\mathsf{Z}_{\lS,\jS}=
[\scal{\mathsf{r}_{\lS}}{\mathsf{e}_{\jS}}-\upxi_{\lS},
\scal{\mathsf{r}_{\lS}}{\mathsf{e}_{\jS}}+\upxi_{\lS}]$. Since the
intersection of these sets is empty, we cannot recover the signal
by solving the associated convex feasibility problem. Instead, our
objective is to solve an instantiation of Problem~\ref{prob:10}
with $\mathsf{p}=\mathsf{MN}$, to wit, 
\begin{equation}
\label{e:exp3.3}
\minimize{\mathsf{x}\in\RR^\mathsf{N}}
{{\upalpha}\|\mathsf{x}\|+
\sum_{\lS=1}^\mathsf{M}\sum_{\jS=1}^\mathsf{N}
\mathsf{d}_{\mathsf{Z}_{\lS,\jS}}\big(\scal{\LS_{\lS}\mathsf{x}}
{\mathsf{e}_{\jS}}\big)},
\end{equation}
where $\upalpha=0.05$. 
Since, for every $\mathsf{x}\in\RR^{\mathsf{N}}$, 
$\upalpha\|\mathsf{x}\|
\leq\upalpha\|\mathsf{x}\|+
\sum_{\lS=1}^\mathsf{M}\sum_{\jS=1}^\mathsf{N}
\mathsf{d}_{\mathsf{Z}_{\lS,\jS}}\big(\scal{\LS_{\lS}\mathsf{x}}
{\mathsf{e}_{\jS}}\big)$, condition \ref{p:6i} in 
Proposition~\ref{p:6} holds. 
In addition, for every $\lS\in\{1,\dots,\mathsf{M}\}$ and every
$\jS\in\{1,\dots,\mathsf{N}\}$, $\mathsf{d}_{\ZS_{\lS,\jS}}$
is real-valued. Hence, condition \ref{p:6iib} in 
Proposition~\ref{p:6} holds as well, which confirms that 
\eqref{e:exp3.3} is an instance of Problem~\ref{prob:10}. We can
thus invoke Corollary~\ref{c:1}. 
The three frameworks of 
Sections~\ref{sec:f1}--\ref{sec:f3} are used to solve
\eqref{e:exp3.3}, where the operator $\boldsymbol{\mathsf{E}}$ in
Proposition~\ref{p:2} is that of Example~\ref{ex:11}. 
Two experiments are conducted: the random variable 
$\boldsymbol{\varepsilon}_0$ produces (a) 1 activation with 1
core; and (b) 8 activations with 8 cores.
Given $\upgamma\in\RPP$, the operators 
$(\prox_{\upgamma\mathsf{d}_{\mathsf{Z}_{\lS,\jS}}})_{1\leq
\lS\leq\mathsf{M},1\leq\jS\leq\mathsf{N}}$ are computed
via \cite[Example~24.28]{Livre1} and $\prox_{\upgamma\|\cdot\|}$ 
via \cite[Example~24.20]{Livre1}. Furthermore, the convolutions are
and the inversions of linear operators are implemented using the 
fast Fourier transform \cite{Andr77}; see 
Example~\ref{ex:inv}\ref{ex:invii}.
As mentioned in Section~\ref{sec:50}, we also compare with:
\begin{itemize}
\item
Algorithm~\eqref{e:72}, which can activate only one operator at
each iteration.
\item
Algorithm~\eqref{e:71}, where the random variable 
$\boldsymbol{\varepsilon}_0$ activates (a) 1;  and (b) 8 indices in
$\{1,\dots,\mathsf{p}\}$ with a uniform distribution at each
iteration.
\end{itemize} 
The solution produced by Framework~2 is shown in
Figure~\ref{fig:sol2}. We display in Figure~\ref{fig:ex3.2} the
normalized error versus execution time. 

\subsection{Overlapping group lasso regression}
\label{sec:52}

We address the overlapping group lasso regression problem of 
\cite{Yuyl13}. Here $\HS=\RR^{\mathsf{N}}$ and $\mathsf{q}$ groups 
of indices $(\mathsf{I}_{\kS})_{1\leq\kS\leq\mathsf{q}}$ in
$\{1,\dots,\mathsf{N}\}$ are present, with
$\bigcup_{\kS=1}^{\mathsf{q}}\mathsf{I}_{\kS}
=\{1,\dots,\mathsf{N}\}$. In addition, for every
$\kS\in\{1,\dots,\mathsf{q}\}$, 
\begin{equation}
\VS_{\kS}\colon\RR^{\mathsf{N}}
\to\RR^{\card\,\mathsf{I}_{\kS}}\colon
\mathsf{x}=(\upxi_{\jS})_{1\leq\jS\leq\mathsf{N}}\mapsto
(\upxi_{\jS})_{\jS\in\mathsf{I}_{\kS}}.
\end{equation}
The goal is to
\begin{equation}
\label{e:gL}
\minimize{\mathsf{x}\in\RR^{\mathsf{N}}}{\frac{\upalpha}{2}
\|\mathsf{A}\mathsf{x}-\mathsf{b}\|^2+\dfrac{1}{\mathsf{q}}
\sum_{\kS=1}^{\mathsf{q}}\|\VS_{\kS}\mathsf{x}\|},
\end{equation}
where $\mathsf{A}\in\RR^{\mathsf{M}\times\mathsf{N}}$,
$\mathsf{b}=(\upbeta_{\lS})_{1\leq\lS\leq\mathsf{M}}\in
\RR^\mathsf{M}$, and $\upalpha\in\RPP$. 
In the experiment, $\mathsf{M}=1200$, $\mathsf{N}=3610$,
$\mathsf{q}=40$, and, as in \cite{Yuyl13},
$\upalpha=5/\mathsf{q}^2$. The entries of $\mathsf{A}$ are i.i.d.
samples from a $\mathcal{N}(1,10)$ distribution. The entries of the
reference vector $\bar{\mathsf{x}}\in\RR^\mathsf{N}$ are
i.i.d. samples from a uniform distribution on $[0,10]$, and
$\mathsf{b}=\mathsf{A}\bar{\mathsf{x}}+\mathsf{w}$, where
$\mathsf{w}\in\RR^\mathsf{M}$ has entries that are i.i.d. samples
from a $\mathcal{N}(0,0.1)$ distribution. We split the term
$\|\mathsf{A}\mathsf{x}-\mathsf{b}\|^2$ into a sum of $30$ blocks 
of $40$ entries each. Finally, the groups are defined by
\begin{align}
(\forall\kS\in\{1,\dots,\mathsf{q}\})\;\;
\mathsf{I}_{\kS}=\{90\kS-89,\ldots,90\kS+10\}.
\end{align}
Let $(\mathsf{a}_{\lS})_{1\leq\lS\leq\mathsf{M}}$ be the rows of
$\mathsf{A}$. Then \eqref{e:gL} is equivalent to  
\begin{equation}
\label{e:gL2}
\minimize{\mathsf{x}\in\RR^{\mathsf{N}}}{\sum_{\kS=1}^{\mathsf{p}}
\mathsf{g}_{\kS}\bigl(\mathsf{L}_{\kS}\mathsf{x}\bigr)},
\end{equation}
where $\mathsf{p}=70$,
\begin{equation}
\label{e:gL22}
(\forall\kS\in\{1,\dots,30\})\;\;
\begin{cases}
\mathsf{L}_{\kS}\colon\RR^{\mathsf{N}}\to\RR^{40}\colon\mathsf{x}
\mapsto\bigl(\scal{\mathsf{x}}{\mathsf{a}_{\lS}}
\bigr)_{40(\kS-1)+1\leq\lS\leq40\kS}\\
\mathsf{g}_{\kS}\colon\RR^{40}\to\RR\colon\mathsf{y}\mapsto
\Frac{\upalpha}{2}\|\mathsf{y}
-(\upbeta_{\lS})_{40(\kS-1)+1\leq\lS\leq40\kS}\|^2,
\end{cases}
\end{equation}
and
\begin{equation}
(\forall\kS\in\{31,\dots,70\})\;\;
\begin{cases}
\mathsf{L}_{\kS}\colon\RR^{\mathsf{N}}\to\RR^{100}\colon\mathsf{x}
\mapsto \VS_{\kS-30}\mathsf{x}\\
\mathsf{g}_{\kS}\colon\RR^{100}\to\RR\colon\mathsf{y}\mapsto
\Frac{1}{\mathsf{q}}\|\mathsf{y}\|.
\end{cases}
\end{equation}
Let $\mathsf{x}=(\upxi_{\lS})_{1\leq\lS\leq\mathsf{N}}\in
\RR^{\mathsf{N}}$ and $\jS\in\{1,\dots,\mathsf{N}\}$. Since 
$\bigcup_{\kS=1}^{\mathsf{q}}\mathsf{I}_{\kS}
=\{1,\dots,\mathsf{N}\}$,
\begin{equation}
\label{e:b1}
\frac{1}{\mathsf{q}}\sum_{\kS=1}^{\mathsf{q}}
\|\VS_{\kS}\mathsf{x}\|
=\frac{1}{\mathsf{q}}\sum_{\kS=1}^{\mathsf{q}}
\|(\upxi_{\lS})_{\lS\in\mathsf{I}_{\kS}}\|
=\frac{1}{\mathsf{q}}\sum_{\kS=1}^{\mathsf{q}}
\sqrt{\sum_{\lS\in\mathsf{I}_{\kS}} |\upxi_{\lS}|^2}
\geq\frac{1}{\mathsf{q}}|\upxi_{\jS}|.
\end{equation}
In turn, 
\begin{equation}
\label{e:b2}
\sum_{\kS=1}^{70}
\mathsf{g}_{\kS}\bigl(\mathsf{L}_{\kS}\mathsf{x}\bigr)=
\|\mathsf{A}\mathsf{x}-\mathsf{b}\|^2+\dfrac{1}{\mathsf{q}}
\sum_{\kS=1}^{\mathsf{q}}\|\VS_{\kS}\mathsf{x}\|
\geq\frac{1}{\mathsf{q}\mathsf{N}}
\sum_{\jS=1}^{\mathsf{N}}|\upxi_{\jS}|
\geq\frac{1}{\mathsf{q}\mathsf{N}}
\sqrt{\sum_{\jS=1}^{\mathsf{N}}|\upxi_{\jS}|^2}
=\frac{1}{\mathsf{q}\mathsf{N}}\|\mathsf{x}\|,
\end{equation}
which ensures that condition \ref{p:6i} in 
Proposition~\ref{p:6} holds. In addition, for every 
$\kS\in\{1,\dots,70\}$, $\mathsf{g}_{\kS}$ is real-valued. Hence,
condition \ref{p:6iib} in Proposition~\ref{p:6} holds as well.
Therefore Proposition~\ref{p:6} guarantees that \eqref{e:gL2} is an
instance of Problem~\ref{prob:10} and we invoke Corollary~\ref{c:1}
to justify the convergence of the algorithms. We employ the three
frameworks of Sections~\ref{sec:f1}--\ref{sec:f3} to solve
\eqref{e:gL}, where the operator $\boldsymbol{\mathsf{E}}$ in
Proposition~\ref{p:2} is that defined in Example~\ref{ex:11}.
Two experiments are conducted: the
random variable $\boldsymbol{\varepsilon}_0$ produces (a) 1
activation with 1 core; and (b) 8 activations with 8 cores.
Given $\upgamma\in\RPP$ and $\mathsf{z}\in\RR^{40}$, we compute the
$\prox_{\upgamma\|\cdot\|}$ via \cite[Example~24.20]{Livre1},
$\prox_{\upgamma\|\cdot-\mathsf{z}\|^2}$ via
\cite[Proposition~24.8(i)]{Livre1}, and the inverse operators
are computed by solving the linear
systems with Example~\ref{ex:inv}\ref{ex:inviii}.
We also compare with:
\begin{itemize}
\item
Algorithm~\eqref{e:72}.
\item
Algorithm~\eqref{e:71}, where the random variable 
$\boldsymbol{\varepsilon}_0$ activates (a) 1;  and (b) 8 indices in
$\{1,\dots,\mathsf{p}\}$ with a uniform distribution at each
iteration.
\end{itemize}
We display in Figure~\ref{fig:ex1} the normalized error versus 
execution time.

\begin{figure}[ht!]
\centering
\begin{tabular}{c@{}c@{}}
\begin{tikzpicture}[scale=0.545]
\definecolor{darkgray176}{RGB}{176,176,176}
\begin{axis}[height=8cm,width=13cm, legend columns=1 
cell align={left}, xmin=0, xmax=400, ymin=-100, ymax=2,
xtick distance=100,
ytick distance=25,
x grid style={darkgray176},xmajorgrids,
y grid style={darkgray176},ymajorgrids,
tick label style={font=\Large}, 
legend cell align={left},
legend style={at={(0.02,0.15)},anchor=west},
axis line style=thick,]
\addplot [ultra thick, ngreen]
table {figures/ex2/F1b1.txt};
\addplot [ultra thick, orng]
table {figures/ex2/F2b1.txt};
\addplot [ultra thick, dblue]
table {figures/ex2/F3b1.txt};
\addplot [ultra thick, dashed, dviolet]
table {figures/ex2/CHb1.txt};
\addplot [ultra thick, dashed, dred]
table {figures/ex2/RNb1.txt};
\end{axis}
\end{tikzpicture}&
\begin{tikzpicture}[scale=0.545]
\definecolor{darkgray176}{RGB}{176,176,176}
\begin{axis}[height=8cm,width=13cm, legend columns=1 
cell align={left}, xmin=0, xmax=60, ymin=-100, ymax=2,
xtick distance=20,
ytick distance=25,
x grid style={darkgray176},xmajorgrids,
y grid style={darkgray176},ymajorgrids,
tick label style={font=\Large}, 
legend cell align={left},
legend style={at={(0.02,0.15)},anchor=west},
axis line style=thick,]
\addplot [ultra thick, ngreen]
table {figures/ex2/F1b8.txt};
\addplot [ultra thick, orng]
table {figures/ex2/F2b8.txt};
\addplot [ultra thick, dblue]
table {figures/ex2/F3b8.txt};
%\addplot [ultra thick, dashed, dviolet]
%table {figures/ex2/CHb1.txt};
\addplot [ultra thick, dashed, dred]
table {figures/ex2/RNb8.txt};
\end{axis}
\end{tikzpicture}\\
\small{(a)}&\small{(b)}
\end{tabular}
\caption{Experiment of Section~\ref{sec:52}. 
Normalized error $20\log_{10}(\|x_{1,\nS}-x_\infty\|/
\|x_{1,0}-x_\infty\|)$ (dB) versus execution time (s).
(a): Block size $1$ with $1$ core. 
(b): Block size $8$ with $8$ cores. 
{\color{ngreen} Green}: Framework~1.
{\color{orng} Orange}: Framework~2. 
{\color{dblue} Blue}: Framework~3
with Example~\ref{ex:11}.
{\color{dviolet} Dashed violet}: Algorithm~\eqref{e:72}.
{\color{dred} Dashed red}: Algorithm~\eqref{e:71}.
}
\label{fig:ex1}
\end{figure}
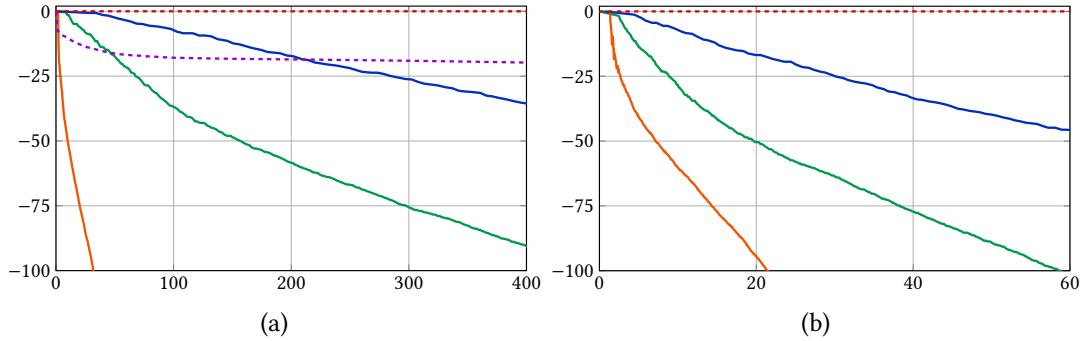

\subsection{Classification using the hinge loss}
\label{sec:53}

We address a binary classification problem. The training data 
samples $(\mathsf{u}_{\kS},\upxi_{\kS})_{1\leq\kS\leq\mathsf{p}}$
are in $\RR^{\mathsf{N}}\times\{-1,1\}$ and the goal is to learn a 
linear classifier $\mathsf{x}\in\HS=\RR^{\mathsf{N}}$. For this
purpose, we solve the instance of Problem~\ref{prob:10}
corresponding to the support vector machine model 
\begin{equation}
\label{e:HL}
\minimize{\mathsf{x}\in\RR^\mathsf{N}}{\frac{\upalpha}{2}
\|\mathsf{x}\|^2+\frac{1}{\mathsf{p}}\sum_{\kS=1}^{\mathsf{p}}
\mathsf{g}_{\kS}(\mathsf{x})},
\end{equation}
where $\upalpha\in\RPP$ and, for every
$\kS\in\{1,\dots,\mathsf{p}\}$, 
$\mathsf{g}_{\kS}\colon\mathsf{x}\mapsto\max\{0,1-\upxi_{\kS}
\scal{\mathsf{x}}{\mathsf{u}_{\kS}}\}$.
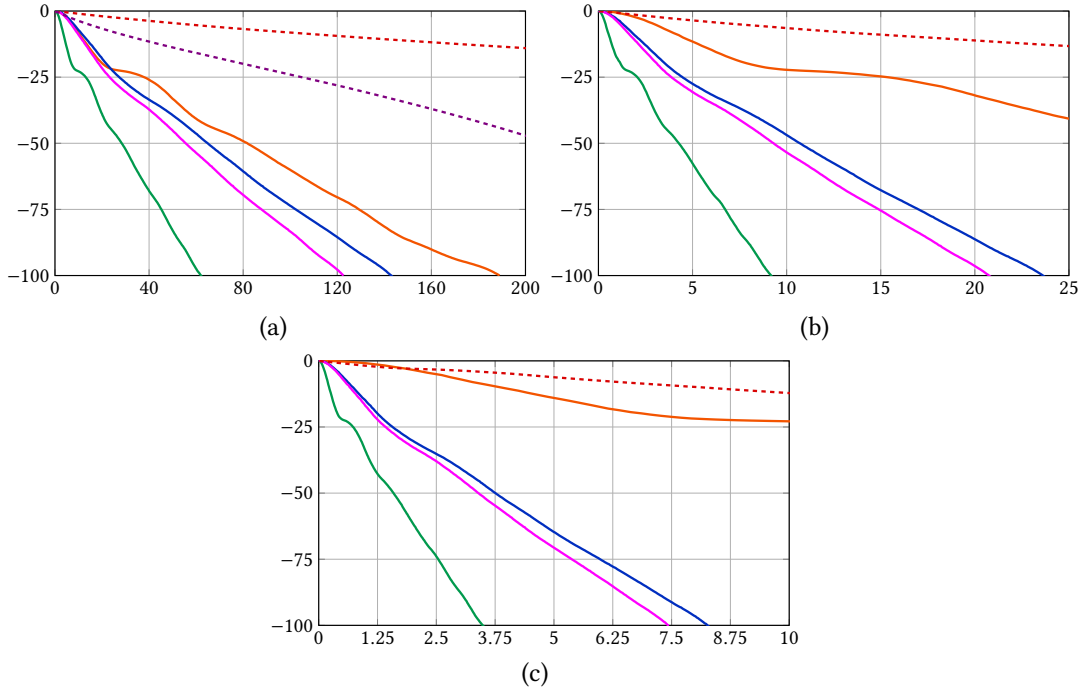
\begin{figure}[b!]
\centering
\begin{tabular}{c@{}c@{}}
\begin{tikzpicture}[scale=0.545]
\definecolor{darkgray176}{RGB}{176,176,176}
\begin{axis}[height=8cm,width=13cm, legend columns=1 
cell align={left}, xmin=0, xmax=200, ymin=-100, ymax=0.00,
xtick distance=40,
ytick distance=25,
x grid style={darkgray176},xmajorgrids,
y grid style={darkgray176},ymajorgrids,
tick label style={font=\Large}, 
legend cell align={left},
legend style={at={(0.02,0.15)},anchor=west},
axis line style=thick,]
\addplot [ultra thick, ngreen]
table {figures/ex3/F1b1.txt};
\addplot [ultra thick, orng]
table {figures/ex3/F2b1.txt};
\addplot [ultra thick, dblue]
table {figures/ex3/F3b1.txt};
\addplot [ultra thick, magenta]
table {figures/ex3/F32b1.txt};
\addplot [ultra thick, dashed, dred]
table {figures/ex3/RNb1.txt};
\addplot [ultra thick, dashed, violet]
table {figures/ex3/CHb1.txt};
\end{axis}
\end{tikzpicture}&%\\
%(a)\\[0.3em]
%\end{figure}
%\begin{figure}
%\centering
\begin{tikzpicture}[scale=0.545]
\definecolor{darkgray176}{RGB}{176,176,176}
\begin{axis}[height=8cm,width=13cm, legend columns=1 
cell align={left}, xmin=0, xmax=25, ymin=-100, ymax=0.00,
xtick distance=5,
ytick distance=25,
x grid style={darkgray176},xmajorgrids,
y grid style={darkgray176},ymajorgrids,
tick label style={font=\Large}, 
legend cell align={left},
legend style={at={(0.02,0.15)},anchor=west},
axis line style=thick,]
\addplot [ultra thick, ngreen]
table {figures/ex3/F1b8.txt};
\addplot [ultra thick, orng]
table {figures/ex3/F2b8.txt};
\addplot [ultra thick, dblue]
table {figures/ex3/F3b8.txt};
\addplot [ultra thick, magenta]
table {figures/ex3/F32b8.txt};
\addplot [ultra thick, dashed, dred]
table {figures/ex3/RNb8.txt};
\end{axis}
\end{tikzpicture}\\
\small{(a)}&\small{(b)}
\end{tabular}
%\end{figure}
%\begin{figure}
%\centering
\begin{tikzpicture}[scale=0.545]
\definecolor{darkgray176}{RGB}{176,176,176}
\begin{axis}[height=8cm,width=13cm, legend columns=1 
cell align={left}, xmin=0, xmax=10, ymin=-100, ymax=0.00,
xtick distance=1.25,
ytick distance=25,
x grid style={darkgray176},xmajorgrids,
y grid style={darkgray176},ymajorgrids,
tick label style={font=\Large}, 
legend cell align={left},
legend style={at={(0.02,0.15)},anchor=west},
axis line style=thick,]
\addplot [ultra thick, ngreen]
table {figures/ex3/F1b32.txt};
\addplot [ultra thick, orng]
table {figures/ex3/F2b32.txt};
\addplot [ultra thick, dblue]
table {figures/ex3/F3b32.txt};
\addplot [ultra thick, magenta]
table {figures/ex3/F32b32.txt};
\addplot [ultra thick, dashed, dred]
table {figures/ex3/RNb32.txt};
\end{axis}
\end{tikzpicture}\\
\small{(c)}
\caption{Experiment of Section~\ref{sec:53}. 
Normalized error $20\log_{10}(\|x_{1,\nS}-x_\infty\|/
\|x_{1,0}-x_\infty\|)$ (dB) versus execution 
time (s).
(a): Block size $1$ with $1$ core. 
(b): Block size $8$ with $8$ cores. 
(c): Block size $32$ with $32$ cores. 
{\color{ngreen} Green}: Framework~1.
{\color{orng} Orange}: Framework~2. 
{\color{nblue} Blue}: Framework~3
with Example~\ref{ex:12}.
{\color{magenta} Magenta}: Framework~3
with Example~\ref{ex:13}.
{\color{dviolet} Dashed violet}: Algorithm~\eqref{e:72}.
{\color{dred} Dashed red}: Algorithm~\eqref{e:71}.
}
\label{fig:ex2}
\end{figure}
In the experiment, $\mathsf{N}=1500$, $\upalpha=1$, 
$\mathsf{p}=750$, and, for every $\kS\in\{1,\dots,\mathsf{p}\}$, 
the entries of $\mathsf{u}_{\kS}$ are i.i.d. samples from a 
$\mathcal{N}(100,10)$ distribution, and
$(\upxi_{\kS})_{1\leq\kS\leq\mathsf{p}}$ are i.i.d. samples from a 
uniform distribution on $\{-1,1\}$. 
Since, for every $\mathsf{x}\in\RR^{\mathsf{N}}$, 
$(\upalpha/2)\|\mathsf{x}\|^2\leq(\upalpha/2)\|\mathsf{x}\|^2+
\sum_{\kS=1}^{\mathsf{p}}\mathsf{g}_{\kS}(\mathsf{x})$, condition
\ref{p:6i} in Proposition~\ref{p:6} holds. 
In addition, for every $\kS\in\{1,\dots,\mathsf{p}\}$, 
$\mathsf{g}_{\kS}$ is real-valued, so that condition \ref{p:6iib}
in Proposition~\ref{p:6} holds as well. This guarantees that 
\eqref{e:HL} is an instance of Problem~\ref{prob:10} and 
we can therefore invoke Corollary~\ref{c:1}. We employ four methods
to solve this problem: Framework~1, Framework~2, and Framework~3
using the operators $\boldsymbol{\mathsf{E}}$ defined in 
Examples~\ref{ex:12} and \ref{ex:13}. 
In the case of Example~\ref{ex:13} in Framework~3, the random
variable $\boldsymbol{\varepsilon}_{0}$ activates indices 
uniformly in $\{1,\dots,2\mathsf{p}+2\}$.
Three experiments are conducted: the
random variable $\boldsymbol{\varepsilon}_0$ produces (a) 1
activation with 1 core; (b) 8 activations with 8 cores, and 
(c) 32 activations with 32 cores.
Given $\upgamma\in\RPP$, the operators 
$(\prox_{\upgamma\mathsf{g}_{\kS}})_{1\leq\kS\leq\mathsf{p}}$ are
computed via \cite[Example~24.37]{Livre1}. 
The inverse operators are explicitly computed in
Example~\ref{ex:inv}\ref{ex:invi}. 

We also compare with:
\begin{itemize}
\item
Algorithm~\eqref{e:72}, which can activate only one operator at
each iteration.
\item
Algorithm~\eqref{e:71}, where the random variable 
$\boldsymbol{\varepsilon}_0$ activates (a) 1; (b) 8;
and (c) 32 indices in $\{1,\dots,\mathsf{p}\}$ with a uniform
distribution at each iteration.
\end{itemize} 
We display in Figure~\ref{fig:ex2} the normalized error versus 
execution time for each instances. The execution time is evaluated 
based on the assumption that the computation corresponding to each 
selected index is assigned to a dedicated core and that all the
cores are working in parallel.

\subsection{Image reconstruction from phase}
\label{sec:54}

In contrast with the previous examples, we consider a data analysis
framework, first proposed in \cite{Siim22}, which requires the
monotone inclusion format of Problem~\ref{prob:1} and is not
reducible to the minimization setting of Problem~\ref{prob:10}.
The goal is to recover
an image in a nonempty closed convex subset
$\mathsf{C}$ of $\HS$ from $\mathsf{p}$ nonlinear observations
$(\mathsf{r}_{\kS})_{1\leq\kS\leq\mathsf{p}}$ produced by Wiener
models, namely, 
\begin{equation}
\label{e:z1}
\text{find}\:\;{\mathsf{x}}\in\mathsf{C}\:\;
\text{such that}\:\;(\forall\kS\in\{1,\ldots,\mathsf{p}\})\:\;
\mathsf{r}_{\kS}=\mathsf{F}_{\kS}(\LS_{\kS}{\mathsf{x}}),
\end{equation}
where each operator $\mathsf{F}_{\kS}\colon\GS_{\kS}\to\GS_{\kS}$
is firmly nonexpansive and each operator 
$\LS_{\kS}\colon\HS\to\GS_{\kS}$ is linear and bounded.
In many instances, the operators
$(\mathsf{F}_{\kS})_{1\leq\kS\leq\mathsf{p}}$ or
$(\LS_{\kS})_{1\leq\kS\leq\mathsf{p}}$ may be imperfectly known or
the model may be corrupted by perturbations and, as a result,
\eqref{e:z1} may not have solutions. A classical approach would be
to relax it into a minimization problem such as the least-squares
model
\begin{equation}
\label{e:z2}
\minimize{\mathsf{x}\in\mathsf{C}}{\sum_{\kS=1}^{\mathsf{p}}
\|\mathsf{F}_{\kS}(\LS_{\kS}\mathsf{x})-\mathsf{r}_{\kS}\|^2}.
\end{equation}
However, because of the nonlinearity of the operators 
$(\mathsf{F}_{\kS})_{1\leq\kS\leq\mathsf{p}}$, the resulting
optimization problem is nonconvex and usually intractable. 
The strategy of \cite{Siim22} consists in relaxing \eqref{e:z1}
into the variational inequality problem
\begin{equation}
\label{e:z3}
\text{find}\:\;\mathsf{x}\in\mathsf{C}\:\;\text{such that}\:\;
(\forall\mathsf{y}\in\mathsf{C})\;\:\sum_{\kS=1}^{\mathsf{p}}
\upalpha_{\kS}\scal{\LS_{\kS}(\mathsf{y}-\mathsf{x})}
{\mathsf{F}_{\kS}(\LS_{\kS}\mathsf{x})-\mathsf{r}_{\kS}}\geq 0,
\end{equation}
where the weights $(\upalpha_{\kS})_{1\leq\kS\leq\mathsf{p}}$ are
in $\RPP$. As shown there,
\eqref{e:z3} is an exact relaxation of \eqref{e:z1} in the sense
that, if \eqref{e:z1} happens to have solutions, they are the same
as those of \eqref{e:z3}. Let us introduce the operators
\begin{equation}
\label{e:z7}
(\forall\kS\in\{1,\ldots,\mathsf{p}\})
\quad\mathsf{B}_{\kS}=\upalpha_{\kS}
\brk{\mathsf{F}_{\kS}-\mathsf{r}_{\kS}},
\end{equation}
which are maximally monotone by \cite[Example~20.30]{Livre1}. Then,
in terms of the normal cone operator of \eqref{e:nc}, \eqref{e:z3}
is equivalent to
\begin{equation}
\label{e:z4}
\text{find}\;\;\mathsf{x}\in\HS\;\;\text{such that}
\;\;0\in\mathsf{N}_{\mathsf{C}}\mathsf{x}+\sum_{\kS=1}^{\mathsf{p}}
\LS_{\kS}^*\brk1{\mathsf{B}_{\kS}(\LS_{\kS}\mathsf{x})}.
\end{equation}
This inclusion problem is now in the format of Problem~\ref{prob:1}
with $\mathsf{A}=\mathsf{N}_{\mathsf{C}}$, which allows us to apply
the algorithms proposed in Sections~\ref{sec:f1}--\ref{sec:f3} to
solve it with guaranteed almost sure convergence of the iterates to
a solution.

\begin{figure}[b!]
\centering
\begin{tabular}{@{}c@{}c@{}c@{}}
\includegraphics[width=4.2cm]{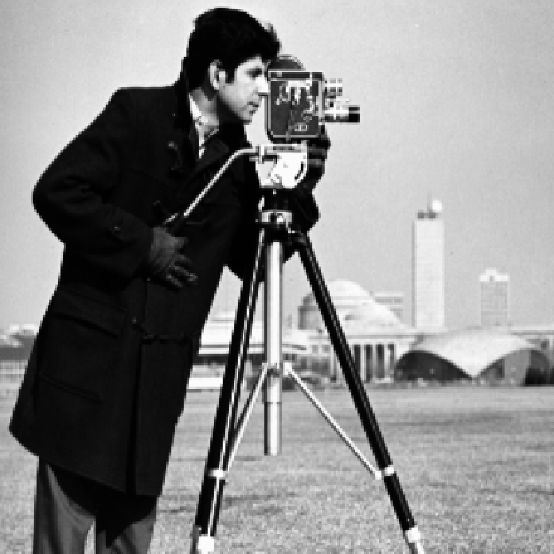}&
\hspace{0.2cm}
\includegraphics[width=4.2cm]{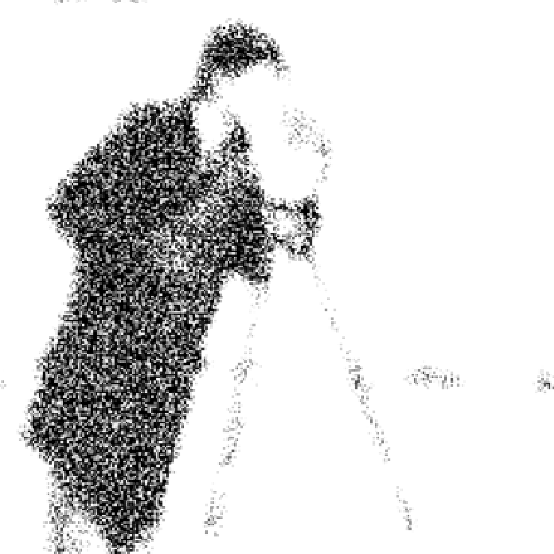}&
\hspace{0.2cm}
\includegraphics[width=4.2cm]{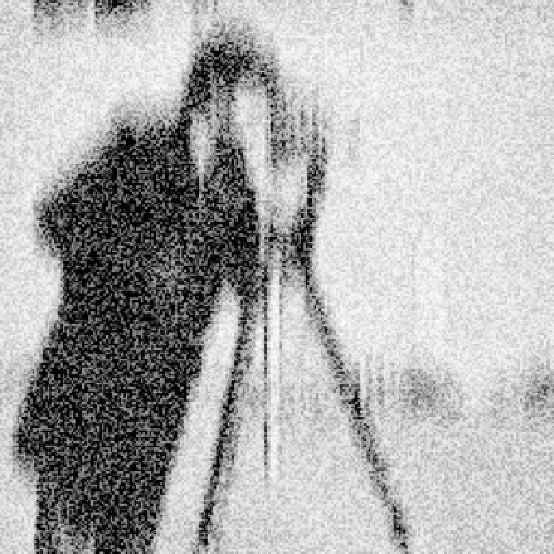}\\
\small{(a)} & \small{(b)} & \small{(c)}
\end{tabular} 
\begin{tabular}{@{}c@{}c@{}}
\includegraphics[width=4.2cm]{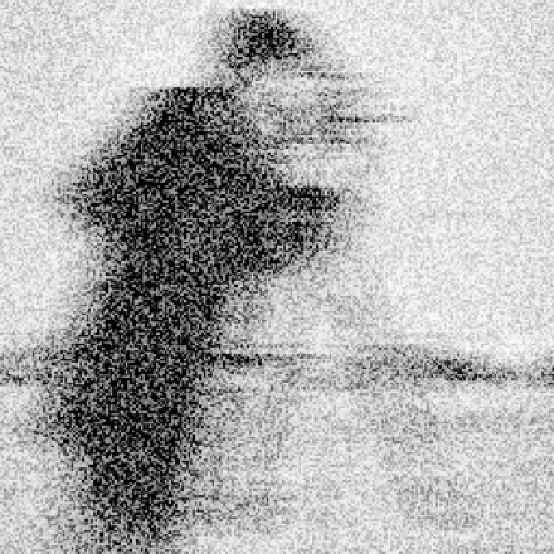}&
\hspace{0.2cm}
\includegraphics[width=4.2cm]{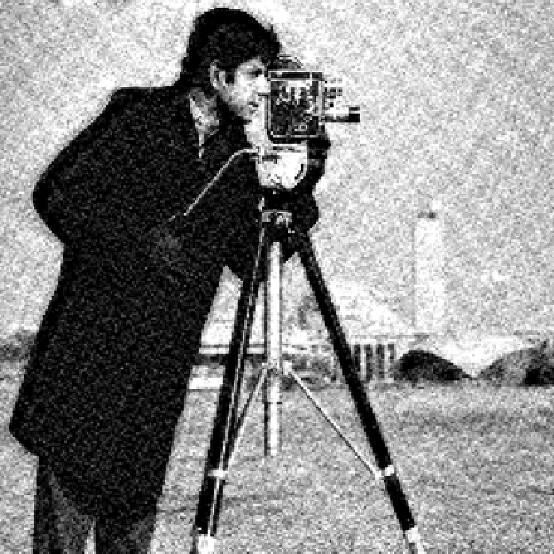}\\
\small{(d)} & \small{(e)} 
\end{tabular} 
\caption{Experiment of Section~\ref{sec:54}:
(a): Original image $\overline{\mathsf{x}}$.
(b): Degraded image $\mathsf{r}_{1}$.
(c): Degraded image $\mathsf{r}_{21}$.
(d): Degraded image $\mathsf{r}_{41}$.
(e): Recovered image.}
\label{fig:e4imag}
\end{figure}

The specific image recovery problem under consideration is similar 
to that of \cite[Section~5.1]{Siim22}. The goal is to recover the 
original image $\overline{\mathsf{x}}\in\HS=\RR^{\mathsf{N}}$ 
($\mathsf{N}=256^2$) of Figure~\ref{fig:e4imag}(a) from the
following prior knowledge and $\mathsf{p}=62$ observations:
\begin{enumerate}
\item
\label{e:o6}
Bounds on pixel values: $\overline{\mathsf{x}}\in\mathsf{C}=
\left[0,255\right]^{\mathsf{N}}$.
\item
\label{e:o1}
The degraded images $(\mathsf{r}_{\kS})_{1\leq\kS\leq 20}$ in
$\RR^{\mathsf{N}}$ are obtained via a blurring process,
addition of noise, and finally clipping.
In terms of the model \eqref{e:z1}, for every
$\kS\in\{1,\dots,20\}$, $\mathsf{G}_{\kS}=\RR^{\mathsf{N}}$, 
$\mathsf{r}_{\kS}=
\mathsf{F}_{\kS}(\LS_{\kS}\overline{\mathsf{x}}+\mathsf{w}_{\kS})$,
where $\LS_{\kS}$ performs convolution with a Gaussian kernel with 
a standard deviation of $3$, 
$\mathsf{w}_{\kS}\in\RR^{\mathsf{N}}$ 
is a noise vector with i.i.d. entries uniformly distributed in
$[-50,50]$, and 
\begin{equation}
\label{e:ex4R1}
\mathsf{F}_{\kS}\colon
\RR^{\mathsf{N}}\to\RR^{\mathsf{N}}\colon
\mathsf{y}\mapsto\proj_{\mathsf{C}_1}\mathsf{y},\quad
\text{where}\quad
\mathsf{C}_1=\left[0,60\right]^{\mathsf{N}}
\end{equation}
models a hard clipping process. This nonlinear measurement process
models a low-quality image acquired by a device which saturates at
photon counts beyond a certain threshold. As an example, the first
degraded image $\mathsf{r}_{1}$ is shown
in Figure~\ref{fig:e4imag}(b). 
\item
\label{e:o2}
The degraded images $(\mathsf{r}_{\kS})_{21\leq\kS\leq40}$ in 
$\RR^{\mathsf{N}}$ are obtained by a process similar to \ref{e:o1}.
Here, for every $\kS\in\{21,\dots,40\}$, the blurring operator 
$\LS_{\kS}$ performs a convolution in the vertical direction with 
a uniform kernel of length $20$, the entries of the
noise vector $\mathsf{w}_{\kS}\in\RR^{\mathsf{N}}$ are i.i.d 
and uniformly distributed in $[-70,70]$, and pixel values beyond 
$90$ are soft-clipped by 
\begin{equation}
\label{e:ex4R2}
\mathsf{F}_{\kS}\colon
\RR^{\mathsf{N}}\to\RR^{\mathsf{N}}\colon
(\upeta_{\jS})_{1\leq\jS\leq\mathsf{N}}
\mapsto\Biggl(\dfrac{90\max\{0,\upeta_{\jS}\}}{90+|\upeta_{\jS}|}
\Biggr)_{1\leq\jS\leq\mathsf{N}}.
\end{equation}
As an example, the degraded image $\mathsf{r}_{21}$ is shown in 
Figure~\ref{fig:e4imag}(c).
\item
The degraded images $(\mathsf{r}_{\kS})_{41\leq\kS\leq60}$ in 
$\RR^{\mathsf{N}}$ are obtained through an image
formation process similar to that of \ref{e:o2}. For every
$\kS\in\{41,\dots,60\}$, the blurring operator $\LS_{\kS}$ now 
performs a convolution in the horizontal direction with a uniform 
kernel of length $24$, and the entries of
the noise vector $\mathsf{w}_{\kS}\in\RR^{\mathsf{N}}$ are i.i.d 
and uniformly distributed in $[-90,90]$. For every
$\kS\in\{41,\dots,60\}$, pixel values beyond $90$ are soft-clipped
by the same operator $\mathsf{F}_{\kS}$ as in \eqref{e:ex4R2}.
\item
The mean pixel value $\uprho=137$ of $\overline{\mathsf{x}}$ is
known. This information is imposed on a candidate solution
$\mathsf{x}\in\RR^{\mathsf{N}}$ via the equation 
$\scal{\mathsf{x}}{\boldsymbol{1}}=\mathsf{N}\uprho$, where
$\boldsymbol{1}=(1,\ldots,1)\in\RR^{\mathsf{N}}$, which corresponds
to the model 
$\mathsf{r}_{61}=\mathsf{F}_{61}(\LS_{61}{\mathsf{x}})$,
with $\GS_{61}=\RR$, $\LS_{61}=\scal{\cdot}{\boldsymbol{1}}$,
$\mathsf{r}_{61}=\mathsf{N}\uprho$, and $\mathsf{F}_{61}=\Id$.
\item
The phase $\uptheta\in\left[-\uppi,\uppi\right]^{\mathsf{N}}$ of
the 2-D discrete Fourier transform of a noise-corrupted version of
$\overline{\mathsf{x}}$, i.e., $\uptheta=\angle
\dft(\overline{\mathsf{x}}+\mathsf{w}_{62})$, where 
$\mathsf{w}_{62}\in\RR^{\mathsf{N}}$ is uniformly distributed in 
$[-3,3]$. This information is enforced by forcing a candidate 
solution to lie in the closed convex set
$\mathsf{C}_{62}=\menge{\mathsf{x}\in\RR^{\mathsf{N}}}
{\angle\dft(\mathsf{x})=\uptheta}$, i.e., by enforcing the
constraint $\mathsf{x}=\proj_{\mathsf{C}_{62}}\mathsf{x}$.
This constraint corresponds to the model 
$\mathsf{r}_{62}=\mathsf{F}_{62}(\LS_{62}{\mathsf{x}})$, with 
$\GS_{62}=\RR^{\mathsf{N}}$, $\LS_{62}=\Id$,
$\mathsf{r}_{62}=\mathsf{0}$, and 
$\mathsf{F}_{62}=\Id-\proj_{\mathsf{C}_{62}}$, that is 
\cite{Youl87},
\begin{equation}
\mathsf{F}_{62}\colon\RR^{\mathsf{N}}\to\RR^{\mathsf{N}}
\colon\mathsf{x}\mapsto \mathsf{x}-
\idft\brk2{\bigl|\dft{\mathsf{x}}\bigr|
\max\bigl\{\cos\bigl(\angle(\dft{\mathsf{x}})-\uptheta\bigr),0
\bigr\}
\exp(\imath\uptheta)}.
\end{equation}
\end{enumerate}
Due to the presence of the measurement errors
$(\mathsf{w}_{\kS})_{1\leq\kS\leq 60}$ and $\mathsf{w}_{62}$, 
problem \eqref{e:z1} is inconsistent and we approximate it by 
\eqref{e:z7}--\eqref{e:z4}, where 
$\upalpha_1=\cdots=\upalpha_{62}=1$.
To implement the algorithms of Sections~\ref{sec:f1}--\ref{sec:f3},
we require the expressions of the resolvent of the operators
$\mathsf{N}_{\mathsf{C}}$ and 
$(\mathsf{B}_{\kS})_{1\leq\kS\leq\mathsf{p}}$. The former is just
\begin{equation}
\mathsf{J}_{\mathsf{N}_{\mathsf{C}}}=\proj_{\mathsf{C}}\colon
(\upxi_{\jS})_{1\leq\jS\leq\mathsf{N}}\mapsto\brk1{
\min\{\max\{0,\upxi_{\jS}\},255\}}_{1\leq\jS\leq\mathsf{N}}.
\end{equation}
For the remaining cases, it follows from \eqref{e:z7} that
the operators $(\mathsf{B}_{\kS})_{1\leq\kS\leq\mathsf{p}}$ are
firmly nonexpansive. We therefore invoke Lemma~\ref{l:600} to
compute their resolvents. Let $\upgamma\in\RPP$ and note that
\cite[Proposition~23.17(ii)]{Livre1} entails that
\begin{equation}
\label{e:cag}
(\forall\kS\in\{1,\ldots,\mathsf{p}\})\quad
\mathsf{J}_{\upgamma\mathsf{B}_{\kS}}=
\mathsf{J}_{\upgamma\mathsf{F}_{\kS}}
(\cdot+\upgamma\mathsf{r}_{\kS}).
\end{equation}
First, set $\kS\in\{1,\dots,20\}$. Then 
$\mathsf{F}_{\kS}=\proj_{\mathsf{C}_1}
=\mathsf{J}_{\mathsf{N}_{\mathsf{C}_1}}$.
Hence, upon setting 
$\mathsf{r}_{\kS}=(\uprho_{\kS,\jS})_{1\leq\jS\leq\mathsf{N}}$, we 
deduce from Lemma~\ref{l:600}\ref{l:600ii} and \eqref{e:cag} that
\begin{equation}
\mathsf{J}_{\upgamma\mathsf{B}_{\kS}} 
\colon(\upxi_{\jS})_{1\leq\jS\leq\mathsf{N}}\mapsto 
\brk3{\upxi_{\jS}+\upgamma\uprho_{\kS,\jS}-
\upgamma\min\Bigl\{\max\Bigl\{0,
\dfrac{\upxi_{\jS}+\upgamma\uprho_{\kS,\jS}}{1+\upgamma}\Bigr\},60
\Bigr\}}_{1\leq\jS\leq\mathsf{N}}.
\end{equation}
On the other hand, for $\kS\in\{21,\dots,60\}$,
$\mathsf{J}_{\upgamma\mathsf{F}_{\kS}}
\colon(\upeta_{\jS})_{1\leq\jS\leq\mathsf{N}}\mapsto
(\upzeta_{\jS})_{1\leq\jS\leq\mathsf{N}}$, where
\begin{equation}
(\forall\jS\in\{1,\dots,\mathsf{N}\})\quad
\upzeta_{\jS}=
\begin{cases}
\dfrac{\upeta_{\jS}-90(1+\upgamma)
+\sqrt{|\upeta_{\jS}-90(1+\upgamma)|^2+360\upeta_{\jS}}}{2},
&\text{if}\;\;\upeta_{\jS}\geq0;\\[3mm]
\upeta_{\jS},&\text{otherwise.}
\end{cases}
\end{equation}
Thus, we derive from \eqref{e:cag} the expressions for
$\mathsf{J}_{\upgamma\mathsf{B}_{\kS}}$. 
Next, we have
$\mathsf{J}_{\upgamma\mathsf{B}_{61}}=(1+\upgamma)^{-1}(\cdot
+\upgamma\mathsf{N}\uprho)$ as a result of
$\mathsf{J}_{\upgamma\mathsf{F}_{61}}=(1+\upgamma)^{-1}\Id$ and
\eqref{e:cag}. Finally, we deduce from 
\cite[Proposition~23.20]{Livre1} that 
\begin{equation}
\label{e:cag2}
\mathsf{F}_{62}=\Id-\proj_{\mathsf{C}_{62}}
=\mathsf{J}_{\mathsf{N}_{\mathsf{C}_{62}}^{-1}}
\quad\text{and}\quad
\mathsf{J}_{(1+\upgamma)^{-1}\mathsf{N}_{\mathsf{C}_{62}}^{-1}}
\circ (1+\upgamma)^{-1}\Id
=\dfrac{\Id-\mathsf{J}_{\mathsf{N}_{\mathsf{C}_{62}}}}{1+\upgamma}
=\dfrac{\Id-{\proj_{\mathsf{C}_{62}}}}{1+\upgamma}.
\end{equation}
Hence, it follows from Lemma~\ref{l:600}\ref{l:600ii} that
$\mathsf{J}_{\upgamma\mathsf{B}_{62}}
=\mathsf{J}_{\upgamma\mathsf{F}_{62}}=(1+\upgamma)^{-1}
(\Id+\upgamma\proj_{\mathsf{C}_{62}})$, i.e.,
\begin{equation}
\mathsf{J}_{\upgamma\mathsf{B}_{62}} 
\colon\mathsf{y}\mapsto\dfrac{\mathsf{y}}{1+\upgamma}
+\dfrac{\upgamma}{1+\upgamma}\idft\brk2{\bigl|\dft{\mathsf{y}}
\bigr|\max\bigl\{\cos\bigl(\angle\bigl(\dft{\mathsf{y}\bigr)-
\uptheta\bigr),0\bigr\}}
\exp(\imath\uptheta)}.
\end{equation}
Lastly, we implement the inversions of linear operators using 
the fast Fourier transform and Example~\ref{ex:inv}\ref{ex:invii}.

We employ the three frameworks of 
Sections~\ref{sec:f1}--\ref{sec:f3} to solve \eqref{e:z4}, 
where Proposition~\ref{p:2} uses the operator
$\boldsymbol{\mathsf{E}}$ defined in Example~\ref{ex:11}. 
Two experiments are conducted: the random variable 
$\boldsymbol{\varepsilon}_0$ produces (a) 1 activation with 1
core; and (b) 8 activations with 8 cores.
We compare with Algorithm~\eqref{e:71}, where the 
random variable $\boldsymbol{\varepsilon}_0$ activates (a) 1;
and (b) 8 indices in $\{1,\dots,\mathsf{p}\}$ with a uniform 
distribution. The solution produced by Framework~3 is shown in
Figure~\ref{fig:e4imag}(e). We display in Figure~\ref{fig:ex4} the
normalized error versus execution time on a single processor
machine.

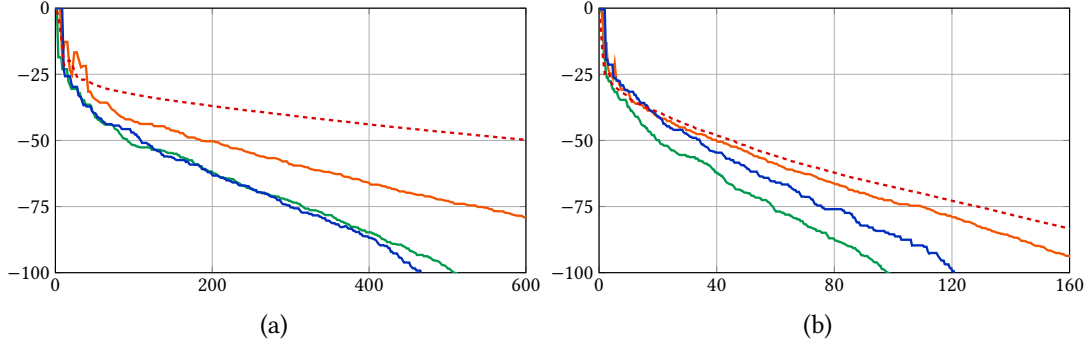
\begin{figure}[ht!]
\centering
\begin{tabular}{c@{}c@{}}
\begin{tikzpicture}[scale=0.545]
\definecolor{darkgray176}{RGB}{176,176,176}
\begin{axis}[height=8cm,width=13cm, legend columns=1 
cell align={left}, xmin=0, xmax=600, ymin=-100, ymax=0.00,
xtick distance=200,
ytick distance=25,
x grid style={darkgray176},xmajorgrids,
y grid style={darkgray176},ymajorgrids,
tick label style={font=\Large}, 
legend cell align={left},
legend style={at={(0.02,0.15)},anchor=west},
axis line style=thick,]
\addplot [ultra thick, ngreen]
table {figures/ex4/F1b1.txt};
\addplot [ultra thick, orng]
table {figures/ex4/F2b1.txt};
\addplot [ultra thick, dblue]
table {figures/ex4/F3b1.txt};
\addplot [ultra thick, dashed, dred]
table {figures/ex4/RNb1.txt};
\end{axis}
\end{tikzpicture}
&
\begin{tikzpicture}[scale=0.545]
\definecolor{darkgray176}{RGB}{176,176,176}
\begin{axis}[height=8cm,width=13cm, legend columns=1 
cell align={left}, xmin=0, xmax=160, ymin=-100, ymax=0.00,
xtick distance=40,
ytick distance=25,
x grid style={darkgray176},xmajorgrids,
y grid style={darkgray176},ymajorgrids,
tick label style={font=\Large}, 
legend cell align={left},
legend style={at={(0.02,0.15)},anchor=west},
axis line style=thick,]
\addplot [ultra thick, ngreen]
table {figures/ex4/F1b8.txt};
\addplot [ultra thick, orng]
table {figures/ex4/F2b8.txt};
\addplot [ultra thick, dblue]
table {figures/ex4/F3b8.txt};
\addplot [ultra thick, dashed, dred]
table {figures/ex4/RNb8.txt};
\end{axis}
\end{tikzpicture}
\\
\small{(a)}&\small{(b)}
\end{tabular}
\caption{Experiment of Section~\ref{sec:54}:
Normalized error $20\log_{10}(\|x_{1,\nS}-x_\infty\|/
\|x_{1,0}-x_\infty\|)$ (dB) versus execution 
time (s).
(a): Block size $1$ with $1$ core. 
(b): Block size $8$ with $8$ cores. 
{\color{ngreen} Green}: Framework~1.
{\color{orng} Orange}: Framework~2. 
{\color{dblue} Blue}: Framework~3
with Example~\ref{ex:11}.
{\color{dred} Dashed red}: Algorithm~\eqref{e:71}.}
\label{fig:ex4}
\end{figure}

\subsection{Discussion}
\label{sec:55}

The three proposed frameworks differ in terms of storage
requirements, use of resolvent operators, and use of linear
operators.
\begin{itemize}
\item
Framework~1: It stores $2\mathsf{p}+3$ vectors.
In addition, for each of the $\mathsf{p}+1$ random activation
indices, there is one resolvent evaluation.
\item
Framework~2: It stores $4\mathsf{p}+5$ vectors. Out of the
$\mathsf{p}+2$ random activation indices, those in
$\{1,\ldots,\mathsf{p}+1\}$ involve the evaluation of a resolvent.
In addition, the linear operators are used only if index 
$\mathsf{p}+2$ is activated.
\item
Framework~3: It stores $2\mathsf{p}+2\mathsf{r}+2$ vectors. 
Moreover, out of the $\mathsf{p}+\mathsf{r}+1$ random activation 
indices, those in $\{1,\ldots,\mathsf{p}+1\}$ involve the 
evaluation of a resolvent operator, while those in 
$\{\mathsf{p}+2,\ldots,\mathsf{p}+\mathsf{r}+1\}$
do not require a resolvent evaluation. 
\end{itemize}

Although Framework~1 is the most efficient in terms of storage, it
may not always be the fastest, especially when resolvents are
computationally expensive. For instance, in Section~\ref{sec:54},
where it is the case, Framework~3 is the fastest. Framework~2 has
an advantage when the linear operators are costly, which is the
case in Section~\ref{sec:52}. Finally we observe that the existing
algorithms \eqref{e:72} and \eqref{e:71} which, as discussed in
Section~\ref{sec:22}, do not satisfy condition 
{\bfseries R2}--{\bfseries R3}, are consistently slower than the
methods proposed in Sections~\ref{sec:f1}--\ref{sec:f3}.

\end{document}